\documentclass{article}
\usepackage{amssymb,amsmath,amsthm,centernot,color,dsfont,enumitem,float,geometry,graphicx,hyperref,mathrsfs,mathtools,pdfpages,stackengine,xparse}
\usepackage[x11names,dvipsnames]{xcolor}
\usepackage{tikz-cd}
\usepackage{pgfplots}
\usepackage{tikz,circuitikz}

\stackMath

\newcommand*{\mint}[1]{%
	\mint@l{#1}{}%
}

\newcommand{\dist}{\operatorname{dist}}

\newcommand{\Hess}{\operatorname{Hess}\,}
\newcommand{\II}{\mathrm{I\!I}}

\newcommand{\R}{\mathbb{R}}

\newcommand{\Ric}{\operatorname{Ric}}

\newcommand{\Span}{\operatorname{span}}

\newcommand{\tr}{\operatorname{tr}}

\newcommand{\vol}{\operatorname{vol}}
\newcommand{\Z}{\mathbb{Z}}

\newtheoremstyle{defn}
{\topsep}   
{\topsep}   
{\color{black}}  
{}          
{\bfseries} 
{.}         
{\newline}         
{}          

\newtheoremstyle{green}
{\topsep}   
{\topsep}   
{\itshape\color{black}}  
{}          
{\bfseries} 
{.}         
{\newline}         
{}          

\newtheoremstyle{black}
{\topsep}   
{\topsep}   
{\color{black}}  
{}          
{\bfseries} 
{.}         
{\newline}         
{}

\theoremstyle{green}
\newtheorem{thm}{Theorem}[section]
\newtheorem{lem}[thm]{Lemma}
\newtheorem{rmk}[thm]{Remark}

\newtheorem{prop}[thm]{Proposition}
\newtheorem{coro}[thm]{Corollary}
\theoremstyle{defn}
\newtheorem{defn}[thm]{Definition}
\theoremstyle{black}

\newtheoremstyle{named}{}{}{\itshape}{}{\bfseries}{.}{.5em}{\thmnote#1{ #3}}
\theoremstyle{named}

\DeclareDocumentCommand{\newfaktor}{m O{0.5} m O{-0.5}}{
	\raisebox{#2\height}{\ensuremath{#1}}
	\mathrel{\scalebox{1.5}{$\diagup$}}
	\raisebox{#4\height}{\ensuremath{#3}}
}

\allowdisplaybreaks
\everymath{\displaystyle}

\newcommand{\Ctilde}{\widetilde{C}}
\newcommand{\gammadot}{\dot{\gamma}}
\newcommand{\gammatilde}{\widetilde{\gamma}}
\newcommand{\gtilde}{\tilde{g}}
\renewcommand{\hbar}{\overline{h}}
\newcommand{\IItilde}{\tilde{\II}}
\newcommand{\Kbar}{\overline{K}}
\newcommand{\Lambdatilde}{\widetilde{\Lambda}}
\newcommand{\LeeWeiRic}{\operatorname{LeWeRic}}
\newcommand{\LinWangRic}{\operatorname{LiWaRic}}

\newcommand{\nablatilde}{\widetilde{\nabla}}
\newcommand{\sigmadot}{\dot{\sigma}}
\newcommand{\umax}{u_{\max}}
\newcommand{\umin}{u_{\min}}

\newcommand{\wexp}{\widetilde{\exp}}
\newcommand{\wRic}{\overline{\Ric}}
\newcommand{\wsec}{\overline{\sec}}

\newtheorem*{soul*}{Soul Theorem}
\newtheorem*{split*}{Splitting Theorem}

\title{Positive Weighted Curvatures on Compact Manifolds with Weighted Convex Boundary and Proper Manifolds}
\author{Ruifeng Xu\footnote{The author's name is also known as Ralph Xu.}}
\begin{document}
\maketitle
\begin{abstract}
Given a Riemannian manifold $(M,g)$ equipped with a positive density function $f = e^{-\varphi}$, we study the weighted sectional curvature $\overline{\sec}_\varphi(U,V)$ and the $k$-th intermediate weighted Ricci curvature $\overline{\text{Ric}}_{k,\varphi}(U,V)$. We show that $\overline{\sec}_\varphi(U,V)$ and $\overline{\text{Ric}}_{k,\varphi}(U,V)$, together with the weighted second fundamental form $\widetilde{\mathrm{I\!I}}$, play central roles in controlling the topology of the manifold. Using the theory of $\widetilde{C}(k)$ functions, variational methods, and Morse theory, we prove that by imposing suitable restrictions on these quantities, an $n$-dimensional compact manifold with boundary or a proper open $n$-dimensional manifold is diffeomorphic to an $n$-ball or to a certain CW complex. This shows that such rigidity phenomena is not confined in the unweighted case, but persist in the weighted setting as well. Lastly, we prove a weighted analogue of Frankel's Theorem under positive intermediate weighted Ricci curvature. These results extend classical theorems from the unweighted setting to the weighted case, demonstrating that $\overline{\sec}_\varphi(U,V)$ and $\overline{\text{Ric}}_{k,\varphi}(U,V)$ provide effective tools in Riemannian geometry.
\end{abstract}

\section{Introduction}
In this paper, we study weighted sectional curvature and $k$-th weighted intermediate Ricci curvature, both of which are defined on a Riemannian manifold $(M,g)$ with a positive real-valued density function $f = e^{-\varphi}$. We call the triple $(M, g, e^{-\varphi})$ a \textit{weighted manifold}.
\subsection{Weighted Sectional Curvature}
The notion of weighted sectional curvature was introduced by Wylie in \cite{Wylie} in 2015. Let $X$ be a vector field defined on $M$. For orthonormal vectors $u$ and $v$, the \textit{weighted sectional curvature} is defined by
$$\wsec_X(u,v) = \sec(u,v) + \dfrac{1}{2}L_X g(v,v) + g(X,v)^2,$$
where $L$ is the Lie derivative with respect to $X$. On weighted manifolds, we often consider the case where $X$ is the gradient of a function $X = \nabla \varphi$, and we denote
$$\wsec_\varphi(u,v) = \sec(u,v) + \Hess \varphi(v,v) + d\varphi(v)^2.$$

Unlike the classical sectional curvature determined by the metric $g$, this quantity is not symmetric in its arguments. In particular,
$$\wsec_\varphi(u,v) \ne \wsec_\varphi(v,u).$$
This notion differs from the sectional curvature associated with the conformal metric $\gtilde = e^{-2\varphi}g$. Moreover, the relationship between the weighted sectional curvatures from $g$ and $\gtilde$ is described in \cite[Proposition 2.1.]{Wylie} as
$$\wsec_\varphi^g(u,v) = e^{-2\varphi}\wsec_{-\varphi}^{\gtilde}(v,u).$$

In \cite{Wylie}, several fundamental results in Riemannian geometry were extended to the weighted setting. In particular, Wylie established a second variation formula of energy involving $\wsec_\varphi$ in place of the classical sectional curvature. The proof of it can be generalized into the second variation formula of arc length using weighted curvatures, which plays a key role in the arguments of this paper.

A geometric meaning for the weighted sectional curvature comes from considering a certain affine connection associated to the weighted manifold $(M, g, e^{-\varphi})$. For any $1$-form $\alpha$, we can define $\nabla^\alpha_X Y = \nabla_X Y - \alpha(X)Y - \alpha(Y)X$, where $\nabla$ is the Levi Civita connection of the metric $g$. In the case of $\alpha = d\varphi$, we denote $\nabla^\alpha = \nabla^\varphi$, so we have
$$\nabla^\varphi_X Y = \nabla_X Y - d\varphi(X)Y - d\varphi(Y)X,$$
which, by \cite[Proposition 3.2.]{WylieYeroshkin}, is the unique torsion-free connection making $e^{-(n+1)\varphi}\vol_g$ parallel while sharing the same geodesics as $\nabla$, up to reparametrization.

It is important to note that using any connection $\bar{\nabla}$, we can define a $(1,3)$ curvature tensor
\begin{equation}\label{RiemannCurvatureTensorDefnFrom2ndDerivatives}
	R^{\bar{\nabla}}(X,Y)Z = \bar{\nabla}_X\bar{\nabla}_Y Z - \bar{\nabla}_Y\bar{\nabla}_X Z - \bar{\nabla}_{[X,Y]} Z
\end{equation}
and a $(0,2)$ Ricci tensor $\Ric^{\bar{\nabla}}(Y,Z) = \tr(X \mapsto R^{\bar{\nabla}}(X,Y)Z)$. We will observe later that the weighted sectional curvature and $(1-n)$-Bakry-Émery Ricci curvature are natural curvature quantities of the connection $\nabla^\varphi$.

Specifically, the following $(1,3)$-tensor, the weighted Riemannian curvature tensor
\begin{equation}\label{WeightedRiemannianCurvatureTensor}
	R^\varphi(X,Y)Z = \nabla^\varphi_X \nabla^\varphi_Y Z - \nabla^\varphi_Y \nabla^\varphi_X Z - \nabla^\varphi_{[X,Y]} Z
\end{equation}
can be used to recover the weighted sectional curvature via
$$g(R^\varphi(u,v)v,u) = \wsec_\varphi(u,v),$$
for orthonormal vectors $u$ and $v$ (see \cite{WylieYeroshkin}).

Weighted sectional curvature is closely related to the $N$-Bakry-Émery Ricci curvature, which has been extensively studied due to its connections with Ricci flow and Perelman’s proof of the Poincaré conjecture. For $N \in \R$ and a vector field $X$, we define the $N$-Bakry-Émery Ricci tensor by
\begin{equation}
	\Ric^N_X = \Ric + \dfrac{1}{2}L_Xg - \frac{X^\# \otimes X^\#}{N}. \label{Ric^N_X Defn}
\end{equation}
Similarly, in the gradient case, where $X = \nabla \varphi$, this reduces to
\begin{equation}
	\Ric^N_\varphi = \Ric + \Hess \varphi - \frac{d\varphi \otimes d\varphi}{N}. \label{Ric^N_varphi Defn}
\end{equation}

The $N$-Bakry-Émery Ricci curvature tensor is a way to generalize the Ricci curvature tensor. Indeed, they are equal when the function $\varphi$ is constant in (\ref{Ric^N_varphi Defn}), or equivalently the vector field $X = 0$ in (\ref{Ric^N_X Defn}). Initially, the Bakry-Émery tensor originates in the work of Lichnerowicz in \cite{Lichnerowicz1} and \cite{Lichnerowicz2}, and it was later developed further by Bakry and Émery in \cite{BakryEmery}. Perelman and Hamilton also discuss it in their papers on Ricci flow in \cite{Hamilton} and \cite{Perelman}. It also shows up in the study of metric measure spaces in \cite{LottVillani}, \cite{Sturm1}, and \cite{Sturm2}. There are too many recent papers on this topic to reference, so we only cite some such as \cite{Lott}, \cite{Morgan1}, \cite{Morgan2}, \cite{MunteanuWang}, \cite{WylieWei}, \cite{Wylie2}, \cite{Wylie3}, \cite{WylieYeroshkin}, and \cite{WoolgarWylie}.

This definition also leads naturally to the notion of $N$-quasi-Einstein manifolds, which are Riemannian manifolds satisfying
$$\Ric_X^N = \lambda g,$$
for some $\lambda \in \R$. Such manifolds play an important role in conformal geometry and in certain models arising in mathematical physics, including general relativity. They have been studied in recent work by Cochran \cite{Cochran}, Valiyakath \cite{Valiyakath}, Lim \cite{Lim}, and many others. When $N = 2$, the equation also arises in near horizon geometry, which is pointed out by \cite{KhuriWoolgar}.

Recall that in the classical setting,
$$\Ric(v,v) = \sum_{i=1}^{n-1} \sec(e_i,v),$$
where $v \in T_pM$ is a unit vector and $\{v, e_1, ..., e_{n-1}\}$ is an orthonormal basis of $T_pM$. Mimicking this construction, one can sum weighted sectional curvatures over a basis. This yields the $(1-n)$-Bakry-Émery Ricci curvature up to a constant re-scaling of the weight function:
\begin{align*}
	\sum_{i = 1}^{n-1} \wsec_\varphi(e_i, v) &= \sum_{i = 1}^{n-1} \sec(e_i, v) + \sum_{i = 1}^{n-1} \Hess \varphi(v,v) + \sum_{i = 1}^{n-1} d\varphi(v)^2\\
	&= \Ric(v,v) + (n-1)\Hess \varphi(v,v) + (n-1)d\varphi(v)^2\\
	&= \Ric(v,v) + \Hess ((n-1)\varphi) (v,v) - \dfrac{d((n-1)\varphi)(v)^2}{1-n}\\
	&= \Ric_{(n-1)\varphi}^{1-n}(v,v),
\end{align*}
where $\{v, e_1, ..., e_{n-1}\}$ are as above. Equivalently, starting with the density function $\dfrac{\varphi}{n-1}$, we obtain
\begin{align*}
	\sum_{i = 1}^{n-1} \wsec_{\frac{\varphi}{n-1}}(e_i, v) &= \sum_{i = 1}^{n-1} \sec(e_i, v) + \sum_{i = 1}^{n-1} \Hess \left(\dfrac{\varphi}{n-1}\right)(v,v) + \sum_{i = 1}^{n-1} d\left(\dfrac{\varphi}{n-1}\right)(v)^2\\
	&= \Ric(v,v) + \Hess \varphi(v,v) - \dfrac{d\varphi(v)^2}{1-n}\\
	&= \Ric_{\varphi}^{1-n}(v,v).
\end{align*}

\subsection{Curvature and Topology}
It is well known that curvature conditions impose strong restrictions on the topology of a manifold. Classical results illustrate this interplay clearly. For instance, the quarter-pinched sphere theorem characterizes manifolds with sufficiently pinched sectional curvature as spheres, while the Bonnet–Myers theorem yields bounds on the diameter and finiteness of the fundamental group under positive Ricci curvature. Such results have been extensively developed for curvature tensors arising from the Levi-Civita connection, but much less is understood about the extent to which curvature quantities associated with non-Levi-Civita connections influence topology. One of the goals of this paper is to investigate this question.

Another fundamental example relating curvature and topology is the Soul Theorem of Cheeger and Gromoll, which states:
\begin{soul*}[(\cite{CheegerGromoll})]
	If $(M,g)$ is a complete noncompact Riemannian manifold with nonnegative sectional curvature, then it contains a closed totally convex submanifold $S \subseteq M$ such that $M$ is diffeomorphic to the normal bundle over $S$.\\
	If $\sec > 0$, then $S$ is a single point, which implies $M$ is diffeomorphic to the Euclidean space.
\end{soul*}

The submanifold $S$ is called the \textit{soul} of $M$. Similarly, Cheeger and Gromoll's Splitting Theorem is also a significant result:
\begin{split*}[(\cite{CheegerGromoll2})]
	If $(M,g)$ contains a line and has $\Ric \ge 0$, then $(M,g)$ is isometric to a product $(H \times \R, g_0 + dt^2)$.
\end{split*}
The generalization of the splitting theorem into weighted settings has been done by many authors. Here, we present a summary of different versions of the splitting theorems, which can also be found in \cite{Lim2}:

\begin{center}
	\underline{If $\Ric_X^N \ge 0$, then:}
	\begin{align*}
		&N>n &\implies &&\text{Splitting Theorem}, \cite{FangLiZhang}, \cite{KhuriWoolgarWylie}\\
		&N<1, X=\nabla\varphi,\varphi<K &\implies &&\text{Splitting Theorem}, \cite{WoolgarWylie}\\
		&N<\infty, X=\nabla\varphi,\varphi<K &\implies &&\text{Splitting Theorem}, \cite{FangLiZhang}\\
		&N=1, X=\nabla\varphi,\varphi<K &\implies &&\text{Warped Product Splitting}, \cite{WoolgarWylie}\\
		&N=\infty, X=\nabla\varphi, \nabla\varphi \to 0 \text{ at } \infty &\implies &&\text{Splitting Theorem}, \cite{FangLiZhang}
	\end{align*}
\end{center}

Although, similar to the Splitting Theorem, many other theorems in classical Riemannian geometry have also been generalized into the weighted setting (see \cite{Wylie}) successfully, the generalization of the Soul Theorem in the weighted setting is still an open question.
%
%
\subsection{Weighted $k$-th Intermediate Ricci Curvature}
Intermediate Ricci curvature arises naturally from the notions of sectional curvature and Ricci curvature, and it provides a quantity that interpolates between the two. Fix a positive integer $k \le n-1$. The \textit{intermediate $k$-th Ricci curvature of $M$} is said to satisfy $\Ric_k(M) > H$ (respectively $\Ric_k(M) < H$) if
$$\sum_{i=1}^k g\left(R(e_i, e_0)e_0, e_i\right) = \sum_{i=1}^k \sec(e_i,e_0) > H \quad (\text{respectively, } < H),$$
for all orthonormal sets $\{e_0, ..., e_k\}$ consisting of $k+1$ vectors.

The study of intermediate Ricci curvature has attracted considerable attention. Wilhelm showed in \cite{Wilhelm} that Synge’s theorem remains valid for complete manifolds satisfying $\Ric_k(M) \ge n-1$, provided the first systole is sufficiently large. More recently, El-Hasan, Phelan, and Wilhelm demonstrated in \cite{El-HasanPhelanWilhelm} that positive $\Ric_k$ is preserved under Riemannian submersions whenever $k \le b-1$, where $b$ denotes the dimension of the codomain. There are also lots of exciting results and discoveries by Mouillé, which we will only reference \cite{Mouille2}, \cite{DominguezVazquezGonzalezAlvaroMouille}, \cite{KennardMouille}, \cite{KhaliliSamaniMouille}, and \cite{KennardMouilleNienhaus}.

In light of the successes in extending notions of section and Ricci curvatures to the weighted setting, in this paper we define a new notion of weighted intermediate Ricci curvature by building on the definition of $\Ric_k(M)$. Fix a positive integer $k \le n-1$. Let $\{v,e_1,...,e_k\}$ be an orthonormal set. The \textit{intermediate weighted $k$-th Ricci curvature of $M$} is said to satisfy $\wRic_{k,\varphi}(M) > H$ (respectively $\wRic_{k,\varphi}(M) < H$) if
$$\sum_{i=1}^k g\left(R^\varphi(e_i, v)v, e_i\right) = \sum_{i=1}^k \overline{\sec}_\varphi(e_i,v) > H \quad (\text{respectively, } < H),$$
for all such orthonormal sets.

Recently, weighted intermediate Ricci curvature has also appeared in \cite{LeeWei} and \cite{LinWang}. Let $P \subseteq M$ be the $k$-dimensional plane spanned by the orthonormal set $\{e_i\}_{i=1}^k$, and suppose that $v \perp P$. Then the (unweighted) $k$-th intermediate Ricci curvature on $P$ is defined as
\begin{equation}
	\Ric_k(v,P) = \sum_{i=1}^{k} \sec(e_i,v).
\end{equation}
Let $\LeeWeiRic_{k,\varphi}^1(v,P)$ denote the $k$-th weighted intermediate Ricci curvature defined in \cite{LeeWei}. Then
\begin{equation}\label{LeeWeiDefn}
	\LeeWeiRic_{k,\varphi}^1(v,P) = \Ric_k(v,P) + \dfrac{k}{n-1}\Hess\varphi(v,v) + \dfrac{k}{(n-1)^2} d\varphi(v)^2.
\end{equation}
Let $\LinWangRic_{k,\varphi}^N(v,P)$ denote the $k$-th intermediate $N$-Bakry--Émery Ricci curvature defined in \cite{LinWang}. Then
\begin{equation}\label{LinWangRicDefn}
	\LinWangRic_{k,\varphi}^N(v,P) = \Ric_k(v,P) + \Hess\varphi(v,v) - \dfrac{1}{N-k} d\varphi(v)^2.
\end{equation}

To better compare our definition with the above definitions, we can also rephrase our definition of a lower (or upper) bound for $\wRic_{k,\varphi}$ using $k$-dimensional planes. Specifically, we can define
\begin{equation}\label{wRicDefnUsingPlanes}
	\wRic_{k,\varphi}(v,P) = \Ric_k(v,P) + k\Hess\varphi(v,v) + k d\varphi(v)^2.
\end{equation}
Then our definition of $\wRic_{k,\varphi}(M) > H$ is equivalent to $\wRic_{k,\varphi}(v,P) > H$ for all $k$-dimensional planes $P$ and all $v \perp P$. An upper bound is defined similarly.

Using \eqref{wRicDefnUsingPlanes}, we can see that our definition and the definition in \cite{LeeWei} are related by
$$\LeeWeiRic_{k,(n-1)\varphi}^1(v,P) = \wRic_{k,\varphi}(v,P).$$
The definition in \cite{LinWang} is a more general form of weighted intermediate Ricci curvature. In the case $N=0$, we have
$$\LinWangRic_{k,\frac{k}{n-1}\varphi}^0(v,P) = \LeeWeiRic_{k,\varphi}^1(v,P) = \wRic_{k,\frac{1}{n-1}\varphi}(v,P).$$

\subsection{Main Results}
Let $M$ be a Riemannian manifold and $N \subseteq M$ be a submanifold. The \textit{weighed second fundamental form} of $N$ is the second fundamental form with respect to the conformal metric $\gtilde$, which can be computed as
$$\IItilde(X,Y) = \II(X,Y) - d\varphi(\mathbf{n})g(X,Y),$$
where $\mathbf{n}$ is a normal vector field to $N$. We denote by $\Lambdatilde_k(x)$ the minimum of all sums of $k$ eigenvalues of $\IItilde$ on $N$. For simplicity, when $k=1$, we write $\Lambdatilde_1 = \Lambdatilde$.

In 1987, Wu showed in \cite{Wu} that suitable bounds on the sectional curvature and the second fundamental form of a compact manifold with nonempty boundary imply that the manifold is diffeomorphic to a closed ball. The following result establishes that an analogous statement holds in the weighted setting.
\begin{thm}\label{WuThm1}
	An $n$-dimensional compact weighted Riemannian manifold $(M,g,e^{-\varphi})$ with non-empty boundary $\partial M$ is diffeomorphic to the standard closed $n$-ball if either of the following condition holds:
	\begin{enumerate}
		\item $\Lambdatilde > 0$ and $\min_{\partial M} \Lambdatilde + \rho_0\min_M \wsec_\varphi > 0$, where $\rho_0$ denotes the maximum distance from $\partial M$.
		\item $\Lambdatilde \ge 0$ and $\wsec_\varphi \ge 0$, and in addition, $\wsec_\varphi > 0$ in a neighborhood of $\partial M$.
		\end{enumerate}
\end{thm}
In particular, in the case where both $\wsec_\varphi > 0$ and $\Lambdatilde > 0$, Theorem \ref{WuThm1} is true.

In the same work, Wu also extended his result to an intermediate curvature setting, showing that the manifold has the homotopy type of a CW complex. This conclusion partially extends to the weighted $k$-th Ricci curvature as well:

\begin{thm}\label{WuThm2}
	Let $(M,g,e^{-\varphi})$ be an $n$-dimensional compact weighted Riemannian manifold with non-empty boundary $\partial M$. Let $1 \le k \le n$ be some integer. Suppose $\Lambdatilde_k \ge 0$ and $\wRic_{k,\varphi} \ge 0$, then
	\begin{enumerate}
		\item $M$ has the homotopy type of a CW complex with a finite number of cells of dimensions $\le k - 1$.
		\item $M$ also has the homotopy type of a CW complex obtained from $\partial M$ by attaching a finite number of cells each of dimension $\ge n - k + 1$.
	\end{enumerate}
\end{thm}

Several years later, in 1993, Zhongmin Shen established in \cite{Shen} that a similar conclusion holds for proper open manifolds, which are not necessarily compact. For the definition of a proper manifold, please refer to Section \ref{ProperManifoldSubsection}. This result also admits a weighted analogue:

\begin{thm}\label{Shen's_Result}
	Let $(M,g,e^{-\varphi})$ be a proper open weighed Riemannian manifold such that $e^\varphi$ grows strictly sublinearly. For some $1 \le k \le n-1$, suppose $\wRic_{k,\varphi} \ge 0$ and $\wRic_{k,\varphi} > 0$ outside of a compact set. Then $M$ has the homotopy type of a CW complex with (possibly infinitely many) cells each of dimension $\le k - 1$. This implies that
	$$H_i(M; \Z) = 0$$
	for all $i \ge k$.
\end{thm}

Note that in the compact case, the weight $e^{\varphi}$ is automatically bounded, whereas in the noncompact setting, the theorem is still true if we assumed $e^{\varphi}$ to be bounded. However, a growth condition called ``strictly sublinear growth'', which will be defined in Definition \ref{StrictlySublinearDefn}, can be applied on $e^{\varphi}$ to ensure the theorem is still true while $e^{\varphi}$ can still be increasing. To avoid repetition, the condition ``strictly sublinear growth'' will be specified in Section \ref{MainResultsOfPositiveWeightedRiccik}.

In the case $k=n-1$, \cite{Lim2} was able to produce similar results without assuming $M$ is proper, but requiring function $\varphi$ to be bounded.

\begin{thm}[\cite{Lim2}] \label{AliceLim'sResult}
	Let $M$ be an $n$-dimensional complete weighted manifold that's noncompact. If $\Ric^N_\varphi > 0$ with $\varphi < K$ for some $K \in \R$ and $N \le 1$, then
	$$H_{n-1}(M,\Z) = 0.$$
\end{thm}

Finally, using the second variation formula (introduced in the next section), we obtain a weighted version of Frankel’s theorem under positive weighted intermediate Ricci curvature:

\begin{thm}\label{Frankel's_Theorem}
	In a complete connected weighted Riemannian manifold $(M,g,e^{-\varphi})$ with $\wRic_{k,\varphi}(M) > 0$, let $N_1, N_2 \subseteq M$ be $r$-dimensional and $s$-dimensional totally geodesic submanifolds with respect to the metric $\gtilde$, respectively. If $s + r \ge n + k - 1$, then
	$$N_1 \cap N_2 \not= \varnothing.$$
\end{thm}

In Section \ref{Preliminary}, we will introduce some preliminary tools such as the second variation formula for arch length and smoothing theorem. In Section \ref{MainResultsOfPositiveWeightedSec}, we will prove Theorem \ref{WuThm1}. In Section \ref{MainResultsOfPositiveWeightedRiccik}, we will introduce proper manifolds, compute necessary bounds the weighted Riemannian curvature tensor, and we will prove Theorem \ref{Shen's_Result}.  In Section \ref{CompactManifoldsWithwRickBounds}, we will prove Theorem \ref{WuThm2}. Lastly, in Section \ref{FrankelsTheorem}, we will prove \ref{Frankel's_Theorem}.

\subsection{Acknowledgment}

The paper author would like to thank his advisor William Wylie for his patience and helpful discussions.

\section{Preliminary}\label{Preliminary}
\subsection{Second Variation Formula}
For a Riemannian manifold $(M,g)$, the classical second variation of arc length formula can be found in \cite[page 300]{Lee}:
$$\dfrac{d^2}{ds^2} L_g(\Gamma_s) = \int_a^b \left|D_t V^\perp \right|^2 - R(V^\perp, \dot\gamma, \dot\gamma, V^\perp) \, dt,$$
where $L_g$ denotes the arc length functional with respect to $g$, $\Gamma(s,t)$ is a proper variation of a geodesic $\gamma$ with variation field $V$, and $V^\perp$ denotes the normal component of $V$.

Note that a second variation formula for the energy functional involving $\wsec_\varphi$ also exists. However, for the purposes of the main results, the arc length formulation is more effective. Following the approach in \cite{Wylie}, we derive a modified second variation formula for arc length in which the curvature tensor $R$ is replaced by $R^\varphi$. This formulation naturally incorporates weighted sectional curvature.

Throughout the remainder of the paper, we will use $\nabla$ to denote the Levi-Civita connection of the metric $g$ and $\nablatilde$ to denote the Levi-Civita connection of the conformal metric $\gtilde = e^{-2\varphi}g$. Additionally, a $\nabla$-geodesic will refer to a curve $\gamma$ satisfying $\nabla_{\gammadot} \gammadot = 0$, and a $\nablatilde$-geodesic will refer to a curve $\gamma$ satisfying $\nablatilde_{\gammadot} \gammadot = 0$. In terms of $\nabla$, $\nablatilde$ has the expression
$$\nablatilde_X Y = \nabla_X Y - d\varphi(X)Y - d\varphi(Y)X + g(X,Y)\nabla\varphi.$$
\begin{prop}[Second Variation of Arc Length]\label{Second_Variation_Formula}
	Suppose $(M,g,e^{-\varphi})$ is a weighted Riemannian manifold. Let $\gamma: [a,b] \to M$ be a unit-speed $\nabla$-geodesic segment, let $\theta : [0,1] \times [a,b] \to M$ be a variation of $\gamma$, and let $V = \dfrac{d\theta}{ds}$ denote its variation field. Then the second variation of the arc length $L_g(\theta_s)$ is given by
	\begin{align*}
		\dfrac{d^2}{ds^2}L_g(\theta_s) 
		=& \int_a^b \big|D_t V^\perp - d\varphi(\dot\gamma)V^\perp\big|^2 
		- g\big(R^\varphi(V^\perp, \dot\gamma)\dot\gamma, V^\perp\big) \, dt \Big|_{s=0} \\
		&+ \Big(d\varphi(\dot\gamma)|V^\perp|^2 + g(D_s V, \dot\gamma)\Big) \Big|_{t=a,\, s=0}^{t=b,\, s=0} \\
		=& \int_a^b \big|D_t V - d\varphi(\dot\gamma)V\big|^2 
		- g(D_t V, \dot\gamma)^2 
		- g\big(R^\varphi(V, \dot\gamma)\dot\gamma, V\big) \, dt \Big|_{s=0} \\
		&+ \Big(d\varphi(\dot\gamma)|V|^2 + g(D_s V, \dot\gamma)\Big) \Big|_{t=a,\, s=0}^{t=b,\, s=0},
	\end{align*}
	where $V^\perp$ denotes the normal component of $V$.
\end{prop}
\begin{proof}
	Denote $S = D_s\theta$ and $T = D_t\theta$. Using the identity
	$$g(R^\varphi(V^\perp, \gammadot)\gammadot,V^\perp) = g(R(V^\perp, \gammadot)\gammadot,V^\perp) + \Hess \varphi (\gammadot, \gammadot)|V^\perp|^2 + df(\gammadot)^2|V^\perp|^2,$$
	and substituting into the classical second variation formula \cite[page 300 to 301]{Lee}, a direct computation yields the first equality.
	\begin{align*}
		\dfrac{d^2}{ds^2} \Bigg|_{s = 0} L_g(\theta_s) =& \int_a^b |D_t V^\perp|^2 - g(R(V^\perp, \gammadot) \gammadot, V^\perp) \; dt \Bigg|_{s=0} + g(D_s V, \gammadot) \Bigg|_{t = a, s = 0}^{t=b, s = 0}\\
		=& \int_a^b |D_t V^\perp|^2 - g(R^\varphi(V^\perp, \gammadot) \gammadot, V^\perp) + \Hess\varphi(\gammadot, \gammadot)|V^\perp|^2 + d\varphi(\gammadot)^2|V^\perp|^2 \; dt \Bigg|_{s = 0}\\
		&+ g(D_s V, \gammadot) \Bigg|_{t=a, s = 0}^{t=b, s = 0}\\
		=& \int_a^b |D_t V^\perp|^2 - g(R^\varphi(V^\perp, \gammadot) \gammadot, V^\perp) + \left(\dfrac{d}{dt}g(\gammadot, \nabla\varphi)\right)|V^\perp|^2 + d\varphi(\gammadot)^2|V^\perp|^2 \; dt \Bigg|_{s=0}\\
		&+ g(D_s V, \gammadot) \Bigg|_{t = a, s = 0}^{t = b, s = 0}\\
		=& \int_a^b |D_t V^\perp|^2 - g(R^\varphi(V^\perp, \gammadot) \gammadot, V^\perp) - g(\gammadot, \nabla\varphi)\dfrac{d}{dt}|V^\perp|^2 + \dfrac{d}{dt}\left(g(\gammadot, \nabla\varphi)|V^\perp|^2\right)\\
		&+ d\varphi(\gammadot)^2|V^\perp|^2 \; dt \Bigg|_{s=0} + g(D_s V, \gammadot) \Bigg|_{t=a, s = 0}^{t=b, s = 0}\\
		=& \int_a^b |D_t V^\perp - d\varphi(\gammadot)V^\perp|^2 - g(R^\varphi(V^\perp, \gammadot) \gammadot, V^\perp) \; dt \Bigg|_{s=0}\\
		&+ \left(d\varphi(\gammadot)|V^\perp|^2 + g(D_s V, \gammadot)\right) \Bigg|_{t=a, s = 0}^{t=b, s = 0}.
	\end{align*}
	This proves the first equality. For the second equality, note that the integrand can be rewritten by decomposing $V$ into components parallel and orthogonal to $\gammadot$, and then expanding and regrouping terms accordingly. This gives the desired expression.
	\begin{align*}
		\dfrac{d^2}{ds^2}L_g(\theta_s)\Bigg|_{s=0} =& \int_a^b |D_t V|^2 - g(D_tV, \gammadot)^2 - g(R(V, \gammadot)\gammadot, V) \; dt \Bigg|_{s = 0} + g(D_sV, \gammadot) \Bigg|_{t = a, s = 0}^{t = b, s = 0}\\
		=& \int_a^b |D_tV - d\varphi(X)V|^2 - g(R^\varphi(V, \gammadot)\gammadot, V) - g(D_tV, \gammadot)^2 \; dt \Bigg|_{s = 0}\\
		&+ \left(d\varphi(\gammadot)|V^\perp|^2 + g(D_s V, \gammadot)\right) \Bigg|_{t=a, s = 0}^{t=b, s = 0}.
	\end{align*}
\end{proof}
\subsection{Smoothing Theorem}
In \cite{Wu}, Wu defined the class of functions $C(k)$ using $\nabla$-geodesics so that $\Hess f(X,X)$ can be expressed in terms of $Cf(x;X)$. We adapt this construction to the setting of $\nablatilde$-geodesics.\\
Let $f: M \to \R$ be a continuous function on a weighted Riemannian manifold $(M,g,e^{-\varphi})$. Let $\gammatilde: (-a,a) \to M$ be a $\nablatilde$-geodesic with $\gammatilde(0) = x \in M$ and $\dot\gammatilde(0) = e^\varphi X \in T_xM$. Define the following extended real numbers:
\begin{align}
	\Ctilde f(x; e^\varphi X) &:= \liminf_{r \to 0} \frac{f(\gammatilde(r)) + f(\gammatilde(-r)) - 2f(\gammatilde(0))}{r^2}, \label{Ctilde(x;e^varphiX)Definition}\\
	\Ctilde f(x) &:= \inf_{X \in T_xM} \Ctilde f(x; e^\varphi X). \label{Ctilde(x)Definition}
\end{align}
Clearly, if $f \in C^2(M)$, then
$$\Ctilde f(x; e^\varphi X) = \Hess_{\gtilde} f(e^\varphi X, e^\varphi X).$$

For any function $f$ (not necessarily $C^2$), we say it is \textit{strictly convex} if there exists a positive function $\eta$ such that
$$\Ctilde f \ge \eta$$
on $M$. Therefore when $f$ is a $C^2$ function, it is \textit{strictly convex} if
$$\Hess_{\gtilde} f > 0.$$
\begin{defn}
	A set of $k$ vectors $\{X_1, ..., X_k\}$ in an inner product space $V$ is \textit{$\epsilon$-orthonormal} if
	$$|\langle X_i, X_j \rangle - \delta_{ij}| < \epsilon,$$
	for all $i, j = 1, ..., k$.\\
	A set of vector fields on a weighted Riemannian manifold is \textit{$\epsilon$-$\gtilde$-orthonormal} if it is $\epsilon$-orthonormal with respect to $\gtilde$ at each point.
\end{defn}

In the context of a weighted Riemannian manifold $(M,g,e^{-\varphi})$, since unit vectors are of the form $e^\varphi X$, we say a set of $k$ vectors $\{X_1, ..., X_k\}$ is $\epsilon$-orthonormal if
$$|\gtilde(e^\varphi X_i, e^\varphi X_j) - \delta_{ij}| < \epsilon.$$
\begin{defn}
	Given a weighted Riemannian manifold $(M, g, e^{-\varphi})$ and an integer $k$, the class of functions $\Ctilde(k)$ on $M$ is defined to be the set of all continuous functions $f$ on $M$ which are Lipschitz continuous on each compact subset of $M$ with the following property:\\
	For each $x_0 \in M$, there exists a neighborhood $W$ of $x_0$ and some $\epsilon, \eta > 0$ such that if $x \in W$ and $\{X_1, ..., X_k\} \subseteq T_xM$ is an $\epsilon$-orthogonal set, then
	$$\sum_{i = 1}^k \Ctilde f(x; e^\varphi X_i) \ge \eta.$$
\end{defn}

It can be verified that the following smoothing lemma introduced in \cite{Wu} remains valid when $C(k)$ is replaced by $\Ctilde(k)$. Since the only difference between the two settings is the choice of geodesics, $\nabla$-geodesics versus $\nablatilde$-geodesics, the arguments carry over.
\begin{lem}[\textbf{Smoothing Theorem for $\Ctilde(k)$}]\label{Smoothing_Thm_for_Ctilde(k)}
	On a weighted Riemannian manifold $(M, g, e^{-\varphi})$, suppose an $f \in \Ctilde(k)$, where $1 \le k \le \dim M$, and a positive continuous function $\xi$ are given, then there exists a $C^\infty$ function $F \in \Ctilde(k)$ such that
	$$|F - f| < \xi.$$
\end{lem}

\section{Compact Manifolds with $\wsec_\varphi(M) > 0$ and Weighted Convex Boundary} \label{MainResultsOfPositiveWeightedSec}
Now we are ready to prove Theorem \ref{WuThm1}. In his original paper \cite{Wu}, Wu used the second variation formula together with the convexity of $-\log \rho$, where $\rho$ is the distance from the boundary $N$ function. In this paper, in order to apply the second variation formula stated in Proposition \ref{Second_Variation_Formula}, we instead use a different convex function. This new function is inspired by \cite{Shen}.
\begin{proof}[Proof to Theorem \ref{WuThm1}:]
	For simplicity, denote $N = \partial M$. Let $x \in B$, where $B \subseteq M \backslash N$ is compact. Let $\rho$ be the distance from the boundary $N$, and let $y \in N$ be such that $\rho(x) = \dist(x,y) = b$ for some $b > 0$. Define
	\begin{equation}
		\chi(\zeta) = \int_{-b}^\zeta \exp\left(\int_{-b}^\xi Z(\tau) \; d\tau \right) \; d\xi, \label{Chi_in_WuThm1}
	\end{equation}
	where $Z: \R \longrightarrow \R$ is a positive continuous function to be chosen later.\\
	We aim to show $\chi(-\rho)$ is convex with respect to the $\gtilde$ metric on some compact $B \subseteq M - N$, i.e. there exists some $\epsilon > 0$ such that
	$$\Ctilde(\chi \circ - \rho) \ge \epsilon$$
	on $B$.\\
	Let $\gamma : [0,b] \longrightarrow M$ be a minimizing normal $\nabla$-geodesic from $x$ to $y$ with $\gamma(0) = x$ and $\gamma(b) = y$. Let $X \in T_xM$ be a $g$-unit vector, and let $X(t)$ be its $\nabla$-parallel transport along $\gamma(t)$. Decompose
	$$X(t) = \alpha X^\perp(t) + \beta \gammadot(t),$$
	where $X^\perp(t) \perp \gammadot(t)$ is a unit vector field along $\gamma(t)$, and $\alpha^2 + \beta^2 = 1$.\\
	Define
	$$W(t) = \alpha X^\perp(t) + \left(1-\dfrac{t}{b}\right) \beta \gammadot(t).$$
	Since $\gamma$ is distance minimizing, we know $\gammadot(b) \perp N$, so $X^\perp(b) \in T_yN$.\\
	Consider the variation
	$$\theta(s,t) = \wexp_{\gamma(t)}(se^{\varphi} W(t)).$$
	For a fixed $s_0$, let $\theta_{s_0}(t) : [0, b] \longrightarrow M$ be the curve $\theta_{s_0}(t) = \theta(s_0,t)$. Define $f(s_0) = L_g(\theta_{s_0}|_0^b)$, the length of $\theta_{s_0}(t)$. At $s_0 = 0$, $\theta_{s_0}(t) = \theta(0,t) = \gamma(t)$, which means $f(0) = b$, so
	$$\chi'(-f(0)) = \chi'(-b) = 1 \text{ and } \chi''(-f(0)) = \chi''(-b) = Z(-b).$$
	Let $h(s) = (\chi \circ -f)(s)$. Then
	\begin{align}
		\notag h''(0) &= \chi''(-f(0))f'(0)f'(0) - \chi'(-f(0)) f''(0).\\
		&= Z(-b)(f'(s))^2 - f''(0). \label{Hess_gtilde h}
	\end{align}
	By the first variation formula,
	\begin{equation}
		f'(0) = \dfrac{d}{ds}\bigg|_{s = 0} L_g(\theta(s,t)) = g\left(e^\varphi W, \gammadot\right) |_{t=0}^{t=b} = -e^\varphi\beta.
	\end{equation}
	Also, by the second variation formula from Proposition \ref{Second_Variation_Formula},
	\begin{align}
		\notag f''(0) =& \dfrac{d^2}{ds^2}\bigg|_{s = 0} L_g(\theta_s)\\
		\notag =& \int_0^b |D_t \left(e^\varphi \alpha X^\perp\right) - d\varphi(\gammadot) e^\varphi \alpha X^\perp|^2 - g(R^\varphi(e^\varphi \alpha X^\perp, \gammadot) \gammadot, e^\varphi \alpha X^\perp) \; dt\\
		&+ \left(d\varphi(\gammadot)|e^\varphi \alpha X^\perp|^2 + g\left(D_s e^\varphi \alpha X^\perp + D_s e^\varphi \beta\left(1-\dfrac{t}{b}\right) \gammadot, \gammadot\right)\right) \Bigg|_{t=0, s = 0}^{t=b, s = 0}. \label{Second_Variation_Plug_in}
	\end{align}
	Since $X^\perp(t)$ is parallel along $\gamma(t)$,
	\begin{equation}
		|D_t \left(e^\varphi \alpha X^\perp\right) - d\varphi(\gammadot) e^\varphi \alpha X^\perp|^2 = |d\varphi(\gammadot) e^\varphi \alpha X^\perp - d\varphi(\gammadot) e^\varphi \alpha X^\perp|^2 = 0. \label{|Stuff|^2=0}
	\end{equation}
	Using the definition of weighted sectional curvature,
	\begin{equation}
		g(R^\varphi(e^\varphi \alpha X^\perp, \gammadot) \gammadot, e^\varphi \alpha X^\perp) = e^{2\varphi} \alpha^2 g(R^\varphi(X^\perp, \gammadot) \gammadot, X^\perp) = e^{2\varphi} \alpha^2 \wsec(X^\perp, \gammadot). \label{Get_wsec}
	\end{equation}
	By the definition of weighted second fundamental form, evaluating the boundary terms of (\ref{Second_Variation_Plug_in}) at $t=b, s=0$ gives
	\begin{align}
		\notag &\left(d\varphi(\gammadot)|e^\varphi \alpha X^\perp|^2 + g\left(D_s e^\varphi \alpha X^\perp + D_s e^\varphi \beta\left(1-\dfrac{t}{b}\right) \gammadot, \gammadot\right)\right) \Bigg|_{t=b, s = 0}\\
		\notag =& d\varphi(\gammadot)|e^{\varphi(y)} \alpha X^\perp|^2 + g\left(\nabla_{\alpha e^{\varphi(y)} X^\perp} e^{\varphi(y)} \alpha X^\perp, \gammadot\right)\\
		=&  -e^{2\varphi(y)} \alpha^2 \tilde{\II}(X^\perp, X^\perp), \label{Evaluate_at_t=b,s=0_to_Get_Weighted_Second_Fundamental_Form}
	\end{align}
	Evaluating the end of (\ref{Second_Variation_Plug_in}) at $t=0, s=0$ gives
	\begin{align}
		\notag&\left(d\varphi(\gammadot)|e^\varphi \alpha X^\perp|^2 + g\left(D_s e^\varphi \alpha X^\perp + D_s e^\varphi \beta\left(1-\dfrac{t}{b}\right) \gammadot, \gammadot\right)\right) \Bigg|_{t=0, s = 0}\\
		\notag =& d\varphi(\gammadot)|e^{\varphi(x)} \alpha X^\perp|^2 + g\left(\nabla_{e^{\varphi(x)} W} e^{\varphi(x)} \alpha X^\perp , \gammadot\right)+ g\left(\nabla_{e^{\varphi(x)} W} e^{\varphi(x)} \beta \gammadot, \gammadot\right)\\
		=& d\varphi(\gammadot)|e^{\varphi(x)} \alpha X^\perp|^2 + g\left(\nabla_{e^{\varphi(x)} W} e^{\varphi(x)} W , \gammadot\right) \label{Evaluate_at_t=0,s=0}.
	\end{align}
	Using the relationship between $\nabla$ and $\nablatilde$,
	\begin{align}
		\notag g\left(\nabla_{e^{\varphi(x)} W} e^{\varphi(x)} W , \gammadot\right) &= g\left(\nablatilde_{e^{\varphi(x)} W} e^{\varphi(x)} W, \gammadot\right) + 2e^{2\varphi(x)}d\varphi(W)g\left(W, \gammadot\right) - d\varphi(\gammadot)\left|e^{\varphi(x)} W\right|^2\\
		\notag &= 2e^{2\varphi(x)}d\varphi(W)g\left(W, \gammadot\right) - d\varphi(\gammadot)\left|e^{\varphi(x)} W\right|^2\\
		\notag &= 2e^{2\varphi(x)}\beta d\varphi(W) - d\varphi(\gammadot)e^{2\varphi(x)}\beta^2 - d\varphi(\gammadot)e^{2\varphi(x)}\alpha^2\\
		&= e^{2\varphi(x)}\beta^2 d\varphi(\gammadot) + 2\alpha\beta e^{2\varphi(x)} d\varphi\left(X^\perp\right) - d\varphi(\gammadot)e^{2\varphi(x)}\alpha^2. \label{Use_Nablatilde_and_Nabla_to_Simplify}
	\end{align}
	Combining results from (\ref{Second_Variation_Plug_in}), (\ref{|Stuff|^2=0}), (\ref{Get_wsec}), (\ref{Evaluate_at_t=b,s=0_to_Get_Weighted_Second_Fundamental_Form}), (\ref{Evaluate_at_t=0,s=0}), and (\ref{Use_Nablatilde_and_Nabla_to_Simplify}), we have
	\begin{align}
		\notag f''(0) =& \int_0^b - e^{2\varphi} \alpha^2 g(R^\varphi(X^\perp, \gammadot) \gammadot, X^\perp) \; dt - e^{2\varphi(y)} \alpha^2 \tilde{\II}(X^\perp, X^\perp)\\
		&- \left(e^{2\varphi(x)}\beta^2 d\varphi(\gammadot) + 2\alpha\beta e^{2\varphi(x)} d\varphi\left(X^\perp\right)\right).
	\end{align}
	Hence, by using (\ref{Hess_gtilde h}),
	\begin{align}
		\notag h''(0) =& Z(-b)e^{2\varphi}\beta^2 + \int_0^b e^{2\varphi} \alpha^2 g(R^\varphi(X^\perp, \gammadot) \gammadot, X^\perp) \; dt + e^{2\varphi(y)} \alpha^2 \tilde{\II}(X^\perp, X^\perp)\\
		\notag &+ e^{2\varphi(x)}\beta^2 d\varphi(\gammadot) + 2\alpha\beta e^{2\varphi(x)} d\varphi\left(X^\perp\right)\\
		\notag =& \alpha^2 \left(\int_0^b e^{2\varphi} \wsec_\varphi(X^\perp, \gammadot) \; dt + e^{2\varphi(y)} \tilde{\II}(X^\perp, X^\perp)\right)\\
		&+ e^{2\varphi(x)}\beta^2(Z(-b) + d\varphi(\gammadot)) + 2\alpha\beta e^{2\varphi(x)} d(X^\perp). \label{Compute_Hess_gtilde h}
	\end{align}
	Completing the square on the last part of (\ref{Compute_Hess_gtilde h}) yields a lower bound
	\begin{align}
		\notag &e^{2\varphi(x)}\beta^2(Z(-b) + d\varphi(\gammadot)) + 2\alpha\beta e^{2\varphi(x)} d(X^\perp)\\
		\notag =& e^{2\varphi(x)}\left(\sqrt{Z(-b) + d\varphi(\gammadot)}\beta + \dfrac{\alpha d\varphi(X^\perp)}{\sqrt{Z(-b) + d\varphi(\gammadot)}}\right)^2 - \dfrac{e^{2\varphi(x)}\alpha^2 d\varphi(X^\perp)^2}{Z(-b) + d\varphi(\gammadot)}\\
		\ge& -\dfrac{e^{2\varphi(x)} \alpha^2 d\varphi(X^\perp)^2}{Z(-b) + d\varphi(\gammadot)}. \label{Complete_the_Square_and_Delete_Positive_Terms}
	\end{align}
	Combining (\ref{Compute_Hess_gtilde h}) and (\ref{Complete_the_Square_and_Delete_Positive_Terms}) gives the inequality
	\begin{equation}\label{Hess_gtildehNonNegative}
		h''(0) \ge \alpha^2 \left(\int_0^b e^{2\varphi} \wsec_\varphi(X^\perp, \gammadot) \; dt + e^{2\varphi(y)} \tilde{\II}(X^\perp, X^\perp) -\dfrac{e^{2\varphi(x)} d\varphi(X^\perp)^2}{Z(-b) + d\varphi(\gammadot)}\right).
	\end{equation}
	Recall that, for any fixed $s_0$, $f(s_0)$ measures the length of $\theta_{s_0}$. Also, $\theta_{s_0}(b) \in N$ because $X^\perp(b) \in T_yN$. Thus, $\theta_{s_0}(t)$ is a path from $\theta_{s_0}(0) \in M$ to the boundary $N$. The length of this segment must be at least $\rho(\theta_{s_0}(0))$, with $-f(0) =-\rho(\theta_0(0))$. Therefore, $-f(s)$ supports \footnote{A function $f_1$ supports another function $f_2$ at $x$ means $f_1(x) = f_2(x)$ and $f_1(y) \le f_2(y)$ for all $y$ near $x$.}\label{support} $-\rho(\theta_s(0))$ at $s=0$ (at $x \in M$), and hence $\chi(-f(s))$ supports $\chi(-\rho(\theta_s(0)))$ at $s=0$.\\
	By (\ref{Ctilde(x;e^varphiX)Definition}), we can compute
	\begin{align*}
		\Ctilde (\chi \circ -f)(x;e^\varphi X) &= \liminf_{s \to 0} \dfrac{\chi(-f(s)) + \chi(-f(-s)) - 2\chi(-f(0))}{x^2}\\
		&\le \liminf_{s \to 0} \dfrac{\chi(-\rho(\theta_s(0))) + \chi(-\rho(\theta_{-s}(0))) - 2\chi(-\rho(\theta_0(0)))}{x^2}\\
		&= \Ctilde (\chi \circ -\rho)(x;e^\varphi X)
	\end{align*}
	By (\ref{Ctilde(x)Definition}), taking infimum over all $X \in T_xM$ gives us
	$$\Ctilde h(x) = \Ctilde (\chi \circ -f)(x) \le \Ctilde (\chi \circ -\rho)(x).$$
	Since $h$ is $C^2$, we have $\Ctilde h(x) = h''(0)$, which is positive by (\ref{Hess_gtildehNonNegative}). Therefore, $\Ctilde(\chi \circ -\rho)$ is also positive. Since $x$ was chosen arbitrarily, it follows from definition that $\chi(-\rho)$ is strictly convex on $M-N$.\\
	The rest of the proof follows from standard Morse theory. Let
	$$\Theta = \int_0^b e^{2\varphi} \wsec_\varphi(X^\perp, \gammadot) \; dt + e^{2\varphi(y)} \tilde{\II}(X^\perp, X^\perp).$$
	We now suppose conditions 1 or 2 holds respectively, and we will show there exists some constant $\psi > 0$ that only depends on $B$ such that $\Theta \ge \psi$.\\
	Suppose condition 1 holds. If $\min_M\wsec < 0$, let
	$$\psi = \begin{cases}
		\umin^2\left(\min_N\Lambdatilde + \rho_0\min_M\wsec_\varphi(M)\right) &\text{ if } \min_M\wsec < 0\\
		\umin^2\min_N \Lambdatilde &\text{ if } \wsec \ge 0,
	\end{cases}$$
	where $\umin^2$ is the minimum of $e^{2\varphi}$ on $B$.\\
	Suppose 2 holds, then for $c > 0$, let $T(c) = \{z \in M \; | \; \rho(z) \le c\}$ be a tubular neighborhood. By our assumption, choose $c < \min_B\rho$ small enough so that $\wsec > \sigma$ for some $\sigma > 0$. Define
	$$\psi = c\sigma.$$
	Under either assumption, it is clear that
	$$\Theta \ge \psi > 0.$$
	Since $B$ is compact, $d\varphi$ is bounded on $B$, so we can choose $Z$ big enough so that
	$$\psi - \dfrac{\umax^2 d\varphi(X^\perp)^2}{Z(-b) + d\varphi(\gamma)} > 0,$$
	where $\umax^2$ is the maximum of $e^{2\varphi}$ on $B$. 
	Therefore $\chi \circ -\rho$ is convex in the metric $\gtilde$.\\
	The rest of the proof follows from Wu's proof in \cite{Wu} as the arguments do not depend on the curvature of the manifold. To summarize, first we apply smoothing theorem from \cite{GreeneWu} to obtain a strictly convex function $\hbar: M-N \longrightarrow \R$ such that $\hbar|_{T(c)-N} = h = \chi \circ -\rho$ for some small enough $c > 0$. Choose $0 < \eta < c$. Define $M' = \{z \in M \; | \; \rho(z) \ge \eta\}$. By \cite[Theorem 3.1.]{Milnor2}, we can conclude $(M, \partial M)$ is diffeomorphic to $(M', \partial M')$. Then we aim to show $(M', \partial M')$ is diffeomorphic to $(B^n, S^{n-1})$.\\
	Note that $h$ is strictly convex with one minimum point, and let it be $x_0$. Rescale $h$ so that $h^{-1}(0) = x_0$ and $h^{-1}(1) = \partial M'$. By Morse Lemma, there exists a coordinate neighborhood $U$ of $x_0$ and a coordinate map $y = (y_1, ..., y_n): U \longrightarrow \R^n$ with $y(x_0) = 0$, and $y$ is a diffeomorphism from $U$ to an open ball $B(2\epsilon) \subseteq \R^n$. Also, $h = \sum_{i = 1}^n y_i^2$ on $U$. By partition of unity, we can change the metric on $M$ to $\bar{g}$ so that $\bar{g}\left(\dfrac{\partial}{\partial y_i}, \dfrac{\partial}{\partial y_j}\right) = \delta_{ij}$ on $U$.\\
	In $U - \{x_0\}$, define $Y = \dfrac{\nabla h}{|\nabla h|^2}$, let $\{\varphi_t\}$ be the flow of $Y$, and let $S = \left\{p \in U \; \Bigg| \; \sum_{i = 1}^n y_i(p)^2 = \epsilon^2 \right\}$. This means $h^{-1}(\epsilon^2) = S$. Hence we can define a diffeomorphism
	$F: S \times (-\epsilon^2,3\epsilon^2) \longrightarrow U - \{x_0\}$ by mapping
	$$(p,t) \mapsto \varphi_t(p),$$
	and the diffeomorphism $y: U - \{x_0\} \longrightarrow B(2\epsilon) - \{0\}$ can be written as
	$$\varphi_t(p) \mapsto \dfrac{\sqrt{\epsilon^2 + t}}{\epsilon}(y_1(p),...,y_n(p)).$$
	$y$ can be extended naturally to a map $M' - \{x_0\} \longrightarrow B^n - \{0\}$ by the uniqueness of integral curves, and allowing $t$ in $F$ to take values in $(-\epsilon^2, 1-\epsilon^2]$. Composing these two functions gives us the desired diffeomorphism $\bar{F}: M' \longrightarrow B^n$.
\end{proof}
\begin{rmk}
	The first part of the proof shows that $h$ is convex on any compact subset $B \subseteq M$. To align with the Morse-theoretic argument, one may take $B = M'$.
\end{rmk}

\section{Proper Manifolds with $\wRic_{k,\varphi}(M) > 0$} \label{MainResultsOfPositiveWeightedRiccik}
We first begin with some preliminary knowledge on Busemann functions and proper manifolds.
\subsection{Background on Proper Manifolds}\label{ProperManifoldSubsection}
On a complete open manifold $M$ with a point $p \in M$, consider the family of functions $b_p^t : M \to \R$ defined by
$$b_p^t(x) = t - d(x, S(p,t)),$$
where $S(p,t)$ denotes the closed geodesic sphere centered at $p$ with radius $t$. By \cite{Shen}, the limit
$$b_p(x) = \lim_{t \to \infty} b_p^t(x)$$
exists and defines the \textit{Busemann function at $p$}.

The following lemma, proved in \cite[Lemma 2.]{Shen}, will be used in the proof of Theorem \ref{Shen's_Result}:
\begin{lem}\label{b_p^q,t_supports_b_p}
	Let $M$ be a complete open manifold and let $p \in M$. For any point $q \in M$, there exists some ray $\sigma_q(t) : [0,\infty) \longrightarrow M$ from $q$ such that the function
	$$b_p^{q,t}(x) = b_p(q) + t - d(x,\sigma_q(t))$$
	supports $b_p(x)$ at $q$.
\end{lem}
Recall that a function $f$ is proper if the preimage of every compact set is compact. As shown in \cite{Shen}, properness of $b_p$ is independent of the base point; that is, $b_p$ is proper if and only if $b_q$ is proper for all $q \in M$. We say that $M$ is \textit{proper} if $b_p$ is proper at some $p \in M$.

Additionally, we will require the following strengthened smoothing result in the proof of Theorem \ref{Shen's_Result}:
\begin{lem}\label{Smoothing_with_proper_Morse_function}
	On a weighted Riemannian manifold $(M, g, e^{-\varphi})$, suppose an $f \in \Ctilde(k)$ is proper, where $1 \le k \le \dim M$, and a positive continuous function $\xi$ are given, then there exists a proper Morse function $F \in \Ctilde(k)$ such that
	$$|F - f| < \xi.$$
\end{lem}
The proof of this lemma is outlined in \cite[page 15]{Shen2}, relying on results from Section 2 of \cite{Milnor}. In the original argument, the setting is a Riemannian manifold equipped with a metric $g$ and its Levi-Civita connection $\nabla$. Replacing these with the conformal metric $\tilde{g}$ and its Levi-Civita connection $\nablatilde$ yields the stated result.
The results from Theorem \ref{WuThm1} can be extended to yield Theorem \ref{WuThm2}. Before proceeding, we first introduce three other key lemmas.\\
It was shown in \cite{CheegerGromoll} that $M$ is proper when $\sec(M) \ge 0$. More recently, progress has been made in \cite{WeiPan}, which provides examples of open manifolds with non-proper Busemann functions. In that paper, it is shown that for $n \ge 4$, there exist manifolds with positive Ricci curvature and non-proper Busemann functions. To our knowledge, it is not known whether $\Ric_k > 0$ implies proper for $1 < k < n-1$.
\subsection{Proper Open Manifolds with $\wRic_{k,\varphi}(M) > 0$}
The following lemma is shown in \cite[Lemma 6.]{Shen}, and the proof of it uses elementary linear algebra. This will be used in the proof of the main theorem.

\begin{lem}\label{Bilinear_form_lemma}
	Let $V$ be an inner product space with dimension $n$. Let $S$ be a symmetric bilinear form on $V$. Suppose for some $k$, where $1 \le k \le n$, and suppose for some $\eta, A > 0$, $S$ satisfies:
	\begin{enumerate}
		\item $\sum_{i = 1}^k S(e_i, e_i) \ge \eta$ for any orthonormal set $\{e_1, ..., e_k\} \subseteq V$.
		\item $|S(v,v)| \le A|v|^2,$ for all $v \in V$.
	\end{enumerate}
	Then there exists some $\epsilon > 0$ depending only on $k$, $\eta$, and $A$ such that for any $\{v_1, ..., v_k\} \subseteq V$ with $|g(v_i, v_j) - \delta{ij}| < \epsilon$, such that
	$$\sum_{i = 1}^k S(v_i, v_i) \ge \dfrac{\eta}{2}.$$
\end{lem}

The next two lemmas concern the symmetries and bounds of the weighted Riemannian curvature tensor defined in (\ref{WeightedRiemannianCurvatureTensor}).

Note that for the usual Riemannian curvature tensor $R$, we always have
$$g(R(U,V)W,U) = g(R(W,U)U,V) = g(R(U,W)V,U)$$
for all vectors $U, V$, and $W$, by \cite[Proposition 3.3.1.]{Petersen}. However, for the curvature tensor $R^\varphi$, these symmetries do not always hold. The following lemma gives conditions under which these symmetries are preserved.

\begin{lem}\label{g(R^varphi(U,V)W,U)=g(R^varphi(U,W)V,U)Andg(R^varphi(U,U)V,U)=0}
	Let $U,V,W$ be vector fields
	\begin{enumerate}
		\item If $U \perp V$ and $U \perp W$, then $g(R^\varphi(U,V)W,U) = g(R^\varphi(U,W)V,U)$.
		\item $g(R^\varphi(U,U)V,U) = 0$.
	\end{enumerate}
\end{lem}
\begin{proof}
	By \cite[Proposition 3.3]{WylieYeroshkin}, for vector fields $X, Y, Z$, the weighted Riemannian curvature tensor $R^\varphi(X,Y)Z$ can be expressed by
	\begin{equation} \label{R^varphiEquationInTermsOfHessAnddvarphi}
		R^\varphi(X,Y)Z = R(X,Y)Z + \Hess\varphi(Y,Z)X - \Hess\varphi(X,Z)Y + d\varphi(Y)d\varphi(Z)X - d\varphi(X)d\varphi(Z)Y.
	\end{equation}
	To show part 1, it follows from (\ref{R^varphiEquationInTermsOfHessAnddvarphi}) that
	\begin{align}
		\notag g(R^\varphi(U,V)W,U) =& g(R(U,V,W),U) + \Hess\varphi(V,W)g(U,U) - \Hess\varphi(U,W)g(U,V) \\
		\notag &+ d\varphi(V)d\varphi(W)g(U,U) - d\varphi(U)d\varphi(W)g(V,U)\\
		=& g(R(U,V,W),U) + \Hess\varphi(V,W)g(U,U) + d\varphi(V)d\varphi(W)g(U,U), \label{g(R^varphi(U,V)W,U)Value}
	\end{align}
	and
	\begin{align}
		\notag g(R^\varphi(U,W)V,U) =& g(R(U,W,V),U) + \Hess\varphi(W,V)g(U,U) - \Hess\varphi(U,V)g(U,W)\\
		\notag &+ d\varphi(W)d\varphi(V)g(U,U) - d\varphi(U)d\varphi(V)g(U,W)\\
		=& g(R(U,W,V),U) + \Hess\varphi(W,V)g(U,U) + d\varphi(W)d\varphi(V)g(U,U). \label{g(R^varphi(U,W)V,U)Value}
	\end{align}
	By (\ref{g(R^varphi(U,V)W,U)Value}) and (\ref{g(R^varphi(U,W)V,U)Value}), we have
	$$g(R^\varphi(U,V)W,U) = g(R^\varphi(U,W)V,U).$$
	To show part 2, it also follows from the definition of $R^\varphi$ (\ref{WeightedRiemannianCurvatureTensor}) that 
	$$R^\varphi(U,U)V = \nabla^\varphi_U\nabla^\varphi_U V - \nabla^\varphi_U\nabla^\varphi_U V - \nabla^\varphi_{[U,U]} V = 0.$$
\end{proof}

\begin{rmk}
	Generally speaking, any curvature tensor $R^{\bar{\nabla}}$ defined by (\ref{RiemannCurvatureTensorDefnFrom2ndDerivatives}) has the property
	$$R^{\bar{\nabla}}(U,U)V = 0$$
	for all vectors $U$ and $V$, simply because
	$$R^{\bar{\nabla}}(U,U)V = \bar{\nabla}_U\bar{\nabla}_U V - \bar{\nabla}_U\bar{\nabla}_U V - \bar{\nabla}_{[U,U]}V = 0.$$
	Part 2 of Lemma (\ref{g(R^varphi(U,V)W,U)=g(R^varphi(U,W)V,U)Andg(R^varphi(U,U)V,U)=0}) is just a specific case of this property.
	One could also prove Part 2 using (\ref{R^varphiEquationInTermsOfHessAnddvarphi}) as follows:
	\begin{align*}
		g(R^\varphi(U,U)V,U) =& g(R(U,U,V),U) + \Hess\varphi(U,V)g(U,U) - \Hess\varphi(U,V)g(U,U)\\
		 &+ d\varphi(U)d\varphi(V)g(U,U) - d\varphi(U)d\varphi(V)g(U,U)\\
		=&0.
	\end{align*}
	However, the property that $R^\varphi(U,U)V = 0$ need not depend on the expression (\ref{R^varphiEquationInTermsOfHessAnddvarphi}).
\end{rmk}

It is also worth noting that for the regular Riemannian curvature tensor, $g(R(U,V)U,U) = 0$ for all vectors $U$ and $V$, because
$$g(R(U,V)U,U) = g(R(U,U)U,V) = 0,$$
by \cite[Proposition 3.3.1.]{Petersen}. For the curvature tensor $R^\varphi$, this symmetry does not hold, even when $U$ and $V$ are orthonormal. The following lemma provides a way to bound this quantity in terms of $\wsec_\varphi$.

\begin{lem}\label{g(R^varphi(U,V)U,U)>=-K}
	Let $(M,g,e^{-\varphi})$ be a weighted Riemannian $n$-manifold with $|\wsec_\varphi| \le K$. Let $U$ and $V$ be orthonormal vector fields. Then
	$$|g(R^\varphi(U,V)U,U)| \le K.$$
\end{lem}

\begin{proof}
	Let $X = \dfrac{U+V}{\sqrt{2}}$ and $Y = \dfrac{U-V}{\sqrt{2}}$. It follows that
	\begin{align*}
		&\wsec_\varphi(Y,X) - \wsec_\varphi(X,Y)\\
		=& \sec(Y,X) + \Hess\varphi(X,X) + d\varphi(X)^2 - \sec(X,Y) - \Hess\varphi(Y,Y) - d\varphi(Y)^2\\
		=& \Hess\varphi(X,X) + d\varphi(X)^2 - \Hess\varphi(Y,Y) - d\varphi(Y)^2\\
		=& \Hess\varphi\left(\dfrac{U+V}{\sqrt{2}},\dfrac{U+V}{\sqrt{2}}\right) + d\varphi\left(\dfrac{U+V}{\sqrt{2}}\right)^2 - \Hess\varphi\left(\dfrac{U-V}{\sqrt{2}},\dfrac{U-V}{\sqrt{2}}\right) - d\varphi\left(\dfrac{U-V}{\sqrt{2}}\right)^2\\
		=& \dfrac{1}{2}\left(\Hess\varphi(U,U) + \Hess\varphi(V,V) + 2\Hess\varphi(U,V)\right) + \dfrac{1}{2}(d\varphi(U)^2 + d\varphi(V)^2 + 2d\varphi(U)d\varphi(V))\\
		&- \dfrac{1}{2}\left(\Hess\varphi(U,U) + \Hess\varphi(V,V) - 2\Hess\varphi(U,V)\right) - \dfrac{1}{2}(d\varphi(U)^2 + d\varphi(V)^2 - 2d\varphi(U)d\varphi(V))\\
		=& 2\Hess\varphi(U,V) + 2d\varphi(U)d\varphi(V)
	\end{align*}
	Therefore
	\begin{equation}\label{Hess+dvarphi^2>-K}
		|\Hess\varphi(U,V) + d\varphi(U)d\varphi(V)| \le K.
	\end{equation}
	By (\ref{R^varphiEquationInTermsOfHessAnddvarphi}), it follows that
	\begin{align}
		\notag g(R^\varphi(U,V)U,U) =& g(R(U,V,U),U) + \Hess\varphi(V,U)g(U,U) - \Hess\varphi(U,U)g(U,V) \\
		\notag &+ d\varphi(V)d\varphi(U)g(U,U) - d\varphi(U)d\varphi(U)g(V,U)\\
		\notag =& g(R(U,V,U),U) + \Hess\varphi(V,U)g(U,U) + d\varphi(V)d\varphi(U)g(U,U)\\
		=& \Hess\varphi(V,U) + d\varphi(V)d\varphi(U) \label{g(R^varphi(U,V)U,U)=Hessvarphi(V,U)+dvarphi(V)dvarphi(U)}
	\end{align}
	Combining (\ref{Hess+dvarphi^2>-K}) and (\ref{g(R^varphi(U,V)U,U)=Hessvarphi(V,U)+dvarphi(V)dvarphi(U)}), we get
	$$|g(R^\varphi(U,V)U,U)| \le K,$$
	as desired.
\end{proof}

The following lemma combines results from Lemma (\ref{g(R^varphi(U,V)W,U)=g(R^varphi(U,W)V,U)Andg(R^varphi(U,U)V,U)=0}) and Lemma (\ref{g(R^varphi(U,V)U,U)>=-K}).

\begin{lem}\label{Estimate_sum_wsec}
	Let $(M, g, e^{-\varphi})$ be a weighted Riemannian $n$-manifold. Suppose $|\overline{\sec}_\varphi|  \le K$ and $\overline{\Ric}_{\varphi, k} \ge H$ at some point $p \in M$, where $1 \le k \le n-1$. Then for any orthonormal set $\{e_1,...,e_k\} \subseteq T_pM$ and for any unit vector $v \in T_pM$,
	$$\sum_{i=1}^{k} g\left(R^\varphi(e_i,v)v,e_i\right) \ge -kK(\alpha^2 + 4\alpha\beta) + H\beta^2,$$
	where $\textstyle\alpha = \sqrt{\sum_{i=1}^k g(v,e_i)^2}$ and $\textstyle\beta = \sqrt{1-\sum_{i=1}^k g(v,e_i)^2}$
\end{lem}
\begin{proof}
	Let $V = \Span\{e_1, ..., e_k\}$, and decompose $v = v_1 + v_2$ with $v_1 \in V$ and $v_2 \perp V$. Then $|v_1| = \alpha^2$ and $|v_2| = \beta^2$.\\
	Choose an orthonormal basis $\{f_1, ..., f_k\}$ for $V$ such that $v_1 = |v_1|f_1$ and let $f_{k+1}$ be a unit vector satisfying $v_2 = |v_2|f_{k+1}$. Then for all $2 \le i \le k$, by Lemma (\ref{g(R^varphi(U,V)W,U)=g(R^varphi(U,W)V,U)Andg(R^varphi(U,U)V,U)=0}) part 1, we have\\
	\begin{equation} \label{When2<=i<=n,commute}
		g(R^\varphi(f_i, f_1)f_{k + 1}, f_i) + g(R^\varphi(f_i, f_{k + 1})f_1, f_i) = 2g(R^\varphi(f_i, f_1)f_{k + 1}, f_i).
	\end{equation}
	Therefore, for $2 \le i \le k$, we have
	\begin{align}
		\notag 2g(R^\varphi(f_i, f_1)f_{k + 1}, f_i) =& g(R^\varphi(f_i, (f_1 + f_{k + 1}))(f_1 + f_{k + 1}), f_i)\\
		&- g(R^\varphi(f_i, f_1)f_1, f_i) - g(R^\varphi(f_i, f_{k + 1})f_{k + 1}, f_i), \label{When2<=i<=n,breakdown}
	\end{align}
	which implies for $2 \le i \le k$,
	$$2|g(R^\varphi(f_i, f_1)f_{k + 1}, f_i)| \le 4K.$$
	When $i = 1$, 
	It follows from Lemma (\ref{g(R^varphi(U,V)W,U)=g(R^varphi(U,W)V,U)Andg(R^varphi(U,U)V,U)=0}) part 2 that
	\begin{equation} \label{g(R^varphi(f1,f1)fk+1,f1)=0}
		g(R^\varphi(f_i, f_1)f_{k + 1}, f_i) = g(R^\varphi(f_1, f_1)f_{k + 1}, f_1) = 0.
	\end{equation} 
	Also, by Lemma \ref{g(R^varphi(U,V)U,U)>=-K}, we know
	\begin{equation}\label{g(R^varphi(f1,fk+1)f1,f1)>=-K}
		g(R^\varphi(f_i, f_{k + 1})f_1, f_i) = g(R^\varphi(f_1, f_{k + 1})f_1, f_1) \ge -K. 
	\end{equation}
	Therefore, combining (\ref{When2<=i<=n,commute}), (\ref{When2<=i<=n,breakdown}), (\ref{g(R^varphi(f1,f1)fk+1,f1)=0}), and (\ref{g(R^varphi(f1,fk+1)f1,f1)>=-K}), we have
	\begin{align*}
		\sum_{i=1}^k g\left(R^\varphi(e_i,v)v,e_i\right) =& \sum_{i=1}^k g\left(R^\varphi(f_i,v)v,f_i\right)\\
		=& |v_1|^2\sum_{i=1}^k g\left(R^\varphi(f_i, f_1)f_1, f_i\right) + |v_1||v_2|\sum_{i=1}^k g\left(R^\varphi(f_i, f_1)f_{k + 1}, f_i\right)\\
		& + |v_1||v_2|\sum_{i=1}^k g\left(R^\varphi(f_i, f_{k + 1})f_1, f_i\right) + |v_2|^2\sum_{i=1}^k g\left(R^\varphi(f_i, f_{k + 1})f_{k + 1}, f_i\right)\\
		=& |v_1|^2\sum_{i=1}^k g\left(R^\varphi(f_i, f_1)f_1, f_i\right) + 2|v_1||v_2|\sum_{i=2}^k g\left(R^\varphi(f_i, f_1)f_{k + 1}, f_i\right)\\
		& + |v_1||v_2|g\left(R^\varphi(f_1, f_{k + 1})f_1, f_1\right) + |v_1||v_2|g\left(R^\varphi(f_1, f_1)f_{k + 1}, f_1\right)\\
		&+ |v_2|^2\sum_{i=1}^k g\left(R^\varphi(f_i, f_{k + 1})f_{k + 1}, f_i\right)\\
		\ge& - (k-1)|v_1|^2K- 4(k-1)|v_1||v_2|K - |v_1||v_2|K + 0 + |v_2|^2H\\
		=& -(k-1)\alpha^2K - (4k-3)\alpha\beta K + H\beta^2\\
		\ge& -k\alpha^2K - 4k\alpha\beta K + H\beta^2\\
		=& -kK(\alpha^2 + 4\alpha\beta) + H\beta^2.
	\end{align*}
\end{proof}

\begin{rmk}
	Even though the sharp bound is $\sum_{i=1}^k g\left(R^\varphi(e_i,v)v,e_i\right) \ge -(k-1)\alpha^2K - (4k-3)\alpha\beta K + H\beta^2$, we choose the smaller lower bound $-kK(\alpha^2 + 4\alpha\beta) + H\beta^2$ to simplify the completing the square process when proving Lemma (\ref{chi_bp_in_Ctilde(k)}).
\end{rmk}

Now we introduce a new growth condition on real-valued functions defined on a Riemannian manifold $M$.

\begin{defn}\label{StrictlySublinearDefn}
	Let $f: M \longrightarrow \R$ be a real valued function on a Riemannian manifold $M$. We say $g$ grows \textit{strictly sublinearly} if there exists some point $p \in M$ such that for all other points $q \in M$, we have
	$$|f(p) - f(q)| < W_1 d(p,q)^j + W_2,$$
	for some $j < 1$ and some constants $W_1 > 0$ and $W_2 \ge 0$.
\end{defn}
If $M = \R$, then in big $O$ notation, $f$ grows strictly sublinearly means
$$f(x) = O(x^c), \text{ for some } 0 \le c < 1.$$

Notice that this condition does not depend on the choice of $p$. For instance, for a fixed point $p' \not= p$, then for all $q \in M$, we must have
\begin{align*}
	|f(p') - f(q)| &= |f(p') - f(p) + f(p) - f(q)|\\
	              &\le |f(p') - f(p)| + |f(p) - f(q)|\\
	              &\le |f(p') - f(p)| + W_1d(p,q)^j + W_2\\
	              &\le W_1d(p,v)^j + W_2'
\end{align*}
where $W_2' = W_2 + |f(p') + f(p)|$. Therefore if $f$ grows strictly sublinearly on $M$, then for any fixed point $p$, we have $|f(p) - f(q)| < W_1d(p,q)^j + W_2$ with $W_1 > 0$ and $W_2 \ge 0$ for all $q \in M$. Please note that this also shows the constants $\alpha$ and $W_1$ do not depend on $p$ and $q$, and it only depends on the function $f$, and the constant $W_2$ only depends on the choice of the fixed point.

With the help of all the preceding lemmas in this subsection, the following lemma immediately implies Theorem \ref{Shen's_Result}.

\begin{lem}\label{chi_bp_in_Ctilde(k)}
	Let $(M, g, e^{-\varphi})$ be a proper open weighted $n$-manifold such that $e^\varphi$ grows strictly sublinearly. Suppose $\wRic_{k,\varphi}(M) \ge 0$ for some $1 \le k \le n-1$ and $\wRic_{k,\varphi} > 0$ outside of a compact subset. Fix some $p \in M$. Then there exists a function $\chi \in C^2(\R)$ such that $\chi \circ b_p$ is a proper function and $\chi \circ b_p$ belongs to $\Ctilde(k)$ on $M$.
\end{lem}
\begin{proof}
	Let $a = \min_{x \in M}b_p(x)$. Choose $R_0 \ge a$ such that $\wRic_{k,\varphi} > 0$ in $\{y \in M \; | \; b_p(y) \ge R_0\}$. For $r \ge a$, define
	\begin{equation}
		H(r) = \inf\{\wRic_{k,\varphi}(y) \; | \; r + R_0 - a \le b_p(y) \le r + R_0 - a + 1\} > 0. \label{Defn_of_H(r)}
	\end{equation}
	Set
	\begin{equation}
		T(r) = \max\left\{\left(\dfrac{8k(W_1 + W_2(j+1))}{(j+1)H(r)I}\right)^{\frac{1}{1-j}}, 2(R_0 - a + 1)\right\}, \label{Defn_of_T(r)}
	\end{equation}
	where $j$, $W_1$, $W_2$ and $I$ are a positive constants to be specified later. Define
	\begin{equation}
		\Kbar(r) = \sup\{|\wsec_\varphi(y)| \; | \; r \le b_p(y) \le r + T(r)\}. \label{Defn_of_Kbar(r)}
	\end{equation}
	Let $Z : [a,\infty) \longrightarrow \R^+$ be a positive continuous function to be determined later, and set
	$$\chi(t) = \int_a^t \exp\left(\int_a^s Z(\tau) \; d\tau\right) \; ds + a.$$
	Then $\chi \in C^2$ and satisfies
	\begin{itemize}
		\item $\chi'(r) \ge 1$ for all $r \in [a, \infty)$.
		\item $\chi''(r) = Z(r)\chi'(r)$ for all $r \in [a,\infty)$.
	\end{itemize}
	Let $q \in M$ with $b_p(q) = r$, and let $\sigma(t)$ be a $\nabla$-ray from $q$. For $V \in T_qM$, Let $V(t)$ be its $\nabla$-parallel transport along $\sigma_q(t)$ with $V(0) = V$.\\
	For all $t \ge 0$ define a $\nablatilde$-geodesic $\theta_{V(t)} : [0,T(r)] \longrightarrow M$ with
	$$\theta_{V(t)}(0) = \left(1-\dfrac{t}{T(r)}\right)e^{\varphi(\sigma_q(t))}V(t).$$
	Define $f: [0,1] \longrightarrow M$ on $\theta_{V(0)}(s)$ as
	\begin{equation}
		f(s) = r + T(r) - \int_0^{T(r)} \left|\dfrac{\partial \theta_{V(t)}}{\partial t}(s)\right| \; dt.
	\end{equation}
	At $q$, we have $s=0$, so $\dfrac{\partial \theta_{V(t)}}{\partial t}(0) = \sigmadot_q(t)$, which means
	\begin{align*}
		f(0) &= r + T(r) - \int_0^{T(r)} \left|\sigmadot_q(t)\right| \; dt\\
		     &= r + T(r) - T(r)\\
		     &= r\\
		     &= b_p(q).
	\end{align*}
	By Lemma \ref{b_p^q,t_supports_b_p}, we know $b_p^{T(r)}(x)$ supports $b_p$ at $q$, which means for $x \not= q$, we have $s = s_x$, for some $0 < s_x \le 1$,
	\begin{align*}
		b_p(x) &\ge b_p^{T(r)}(x)\\
		       &= r + T(r) - d(x,\sigma_q(T_r))\\
		       &\ge r + T(r) - \int_0^{T(r)} \left|\dfrac{\partial \theta_{V(t)}}{\partial t}(s_x)\right| \; dt\\
		       &= f(s_x).
	\end{align*}
	This is true because $\int_0^{T(r)} \left|\dfrac{\partial \theta_{V(t)}}{\partial t}(s_x)\right|$ is the arc length of some segment from $x$ to $\sigma_q(T(r))$, which is always greater than or equal to the true distance $d(x,\sigma_q(T_r))$.
	Hence, this proves $f$ supports $b_p$ at $q$.\\
%
	Let $\{e_1,...,e_j\} \subseteq T_qM$ be an orthonormal set. Let $\alpha,\beta > 0$ with $\alpha^2 + \beta^2 = 1$ where
	$$\alpha^2 = \sum_{i = 1}^k g\left(e_i,\sigmadot_q(0)\right)^2.$$
	From the first variation formula,
	\begin{align}
	\notag Z(r)\sum_{i = 1}^k |e^\varphi e_jf|^2 =& Z(r)\sum_{i = 1}^k |d(f \circ \wexp_{\sigma(t)})(e_j)|^2\\
		\notag =& Z(r)\sum_{i = 1}^k \left|\dfrac{d}{ds}\bigg|_{s=0} - \int_0^{T(r)} \left|\dfrac{\partial \theta_{e_i(t)}}{\partial t}(s)\right| \, dt\right|^2\\ 
		\notag =& Z(r)\sum_{i = 1}^k \left|g(\dot{\sigma}_q(0),e^\varphi e_i)\right|^2\\ 
		=& Z(r)\alpha^2e^{2\varphi} \label{Z(r)sum_of|first_derivatives|^2_result}.
	\end{align}
	By the second variation formula from Proposition \ref{Second_Variation_Formula},
	\begin{align}
		\notag &\sum_{i = 1}^k \Hess_{\gtilde} f(e_i, e_i)\\
		\notag =& \sum_{i = 1}^k \dfrac{d^2}{ds^2}\Bigg|_{s=0} - \int_0^{T(r)} \left|\dfrac{\partial \theta_{e_i}}{\partial s}(t,s)\right| \, dt\\
		\notag =& \sum_{i = 1}^k \int_0^{T(r)} -\left|-\dfrac{1}{T(r)}e^\varphi e_i + \left(1 - \dfrac{t}{T(r)}\right)d\varphi(\sigmadot)e^\varphi e_i - d\varphi(\sigmadot)\left(1 - \dfrac{t}{T(r)}\right)e^\varphi e_i\right|^2 \; dt\\
		\notag &+ \sum_{i = 1}^k \int_0^{T(r)} g\left(R^\varphi\left(\left(1 - \dfrac{t}{T(r)}\right)e^\varphi e_i, \sigmadot\right) \sigmadot, \left(1 - \dfrac{t}{T(r)}\right)e^\varphi e_i\right) \; dt\\
		\notag &+\sum_{i = 1}^k \int_0^{T(r)}g\left(D_t\left(\left(1 - \dfrac{t}{T(r)}\right)e^\varphi e_i\right), \sigmadot\right)^2 \; dt\\
		&+ \sum_{i = 1}^k d\varphi(\sigmadot(0))e^{2\varphi(q)} + \sum_{i = 1}^k g\left(\nabla_{e^{\varphi(q)} e_i}e^{\varphi(q)} e_i, \sigmadot(0) \right). \label{After_Use_2nd_Var_Formula_wRic_k}
	\end{align}
	Since $e^\varphi$ grows strictly sublinearly, we know there exists some constants $W_1 > 0$, $W_2 \ge 0$, and $0 < j < 1$ such that
	\begin{align*}
		|e^{2\varphi(\sigma(t))} - e^{2\varphi(\sigma(0))}| &\le W_1t^j + W_2;\\
		|e^{2\varphi(\sigma(t))}| - |e^{2\varphi(\sigma(0))}|&\le W_1t^j + W_2;\\
		|e^{2\varphi(\sigma(t))}| &\le W_1t^j + W_2 + |e^{2\varphi(\sigma(0))}|,
	\end{align*}
	for all $t \ge 0$. Therefore without loss of generality, we can assume $W_2$ to be large enough so that
	\begin{equation}
		|e^{2\varphi(\sigma(t))}| \le W_1t^j + W_2 + |e^{2\varphi(\sigma(0))}|, \label{Sublinear_inequality}
	\end{equation}
	for all $t \ge 0$.\\
	By the definition of $T(r)$ in (\ref{Defn_of_T(r)}) and (\ref{Sublinear_inequality}), we have
	\begin{align}
		\notag &\sum_{i = 1}^k \int_0^{T(r)} -\left|-\dfrac{1}{T(r)}e^\varphi e_i + \left(1 - \dfrac{t}{T(r)}\right)d\varphi(\sigmadot)e^\varphi e_i - d\varphi(\sigmadot)\left(1 - \dfrac{t}{T(r)}\right)e^\varphi e_i\right|^2 \; dt\\
		\notag =& \sum_{i = 1}^k -\int_0^{T(r)} \dfrac{1}{T(r)^2}e^{2\varphi} \; dt\\
		\notag =& -\dfrac{k}{T(r)^2} \int_0^{T(r)} e^{2\varphi} \; dt\\
		\notag \ge& -\dfrac{k}{T(r)^2} \int_0^{T(r)} W_1t^j + W_2 \; dt\\
		\notag = & -\dfrac{k}{T(r)^2} \left(\dfrac{W_1T(r)^{j+1}}{j+1} + W_2T(r)\right)\\
		\notag = & -\left(\dfrac{kW_1T(r)^{j+1}}{T(r)^2(j+1)} + \dfrac{kW_2}{T(r)}\right)\\
		\notag = & -\left(\dfrac{kW_1}{T(r)^{1-j}(j+1)} + \dfrac{kW_2}{T(r)}\right)\\
		\notag \ge & -\left(\dfrac{kW_1}{T(r)^{1-j}(j+1)} + \dfrac{kW_2(j+1)}{T(r)^{1-j}(j+1)}\right)\\
		\notag = & -\dfrac{1}{T(r)^{1-j}}\left(\dfrac{kW_1+kW_2(j+1)}{j+1}\right)\\
		=& -\dfrac{1}{8}H(r)I, \label{Integral_of_||^2=-H(r)I(r)/W(r)}
	\end{align}
	By Lemma \ref{Estimate_sum_wsec}, we have
	\begin{align}
		\notag &\sum_{i = 1}^k \int_0^{T(r)} g\left(R^\varphi\left(\left(1 - \dfrac{t}{T(r)}\right)e^\varphi e_i, \sigmadot\right) \sigmadot, \left(1 - \dfrac{t}{T(r)}\right)e^\varphi e_i\right) \; dt\\
		\notag =& \sum_{i = 1}^k \int_0^{T(r)} \left(1 - \dfrac{t}{T(r)}\right)^2e^{2\varphi} g(R^\varphi( e_i, \sigmadot)\sigmadot, e_i) \; dt\\
		\notag \ge& -k \int_0^{T(r)} \left(1 - \dfrac{t}{T(r)}\right)^2e^{2\varphi} \max|\wsec_\varphi(\sigma(t))|\left(\alpha^2 + 4\alpha\beta\right) \; dt\\
		&+ \int_0^{T(r)} \left(1 - \dfrac{t}{T(r)}\right)^2e^{2\varphi} \min\left(\wRic_{k,\varphi}(\sigma(t))\right)\beta^2 \; dt. \label{Sum_wsec_estimates_using_wsec_and_wRic}
	\end{align}
	By (\ref{Sublinear_inequality}) and the definition of $\Kbar(r)$ in (\ref{Defn_of_Kbar(r)}), we have
	\begin{align}
		\notag &-k \int_0^{T(r)} \left(1 - \dfrac{t}{T(r)}\right)^2e^{2\varphi} \max|\wsec_\varphi(\sigma(t))|\left(\alpha^2 + 4\alpha\beta\right) \; dt\\
		\notag \ge& -k\Kbar(r)\left(\alpha^2 + 4\alpha\beta\right) \int_0^{T(r)} \left(1 - \dfrac{t}{T(r)}\right)^2e^{2\varphi} \; dt\\
		\notag \ge& -k\Kbar(r)\left(\alpha^2 + 4\alpha\beta\right) \int_0^{T(r)} \left(1 - \dfrac{t}{T(r)}\right)^2(W_1t^j + W_2) \; dt\\
		=& -k\Kbar(r)\left(\alpha^2 + 4\alpha\beta\right) \left(\dfrac{2W_1T(r)^{j+1}}{(j+1)(j+2)(j+3)} + \dfrac{W_2T(r)}{3}\right). \label{wsec_integral_estimates}
	\end{align}
	By the definition of $H(r)$ in (\ref{Defn_of_H(r)}) and $T(r)$ in (\ref{Defn_of_T(r)}), we have
	\begin{align}
		\notag &\int_0^{T(r)} \left(1 - \dfrac{t}{T(r)}\right)^2e^{2\varphi} \min\left(\wRic_{k,\varphi}(\sigma(t))\right)\beta^2 \; dt\\
		\notag \ge& H(r)\beta^2 \int_{R_0 - a}^{R_0 - a + 1} \left(1 - \dfrac{t}{T(r)}\right)^2 e^{2\varphi} \; dt\\
		\ge& H(r)\beta^2 \int_{R_0 - a}^{R_0 - a + 1} \left(1 - \dfrac{t}{2(R_0 - a + 1)}\right)^2 e^{2\varphi} \; dt. \label{(1-t/T)^2_ge_(1-t/2(R_0-a+1))^2}
	\end{align}
	Note that $\left(1 - \dfrac{t}{2(R_0 - a + 1)}\right)^2$ is decreasing on $[0,2(R_0 - a + 1)]$, we know
	\begin{equation}
		\left(1 - \dfrac{t}{2(R_0 - a + 1)}\right)^2 \ge \left(1 - \dfrac{R_0 - a + 1}{2(R_0 - a + 1)}\right)^2 = \dfrac{1}{4}. \label{(1-t/2(R_0-a+1))^2_bounded_below_by_1/4}
	\end{equation}
	Combining (\ref{(1-t/T)^2_ge_(1-t/2(R_0-a+1))^2}) and (\ref{(1-t/2(R_0-a+1))^2_bounded_below_by_1/4}), and by the fact that $e^{2\varphi} > 0$, we know
	\begin{align}
		\notag \int_0^{T(r)} \left(1 - \dfrac{t}{T(r)}\right)^2e^{2\varphi} \min\left(\wRic_{k,\varphi}(\sigma(t))\right)\beta^2 \; dt &\ge \dfrac{1}{4}H(r)\beta^2 \int_{R_0 - a}^{R_0 - a + 1} e^{2\varphi} \; dt\\
		&= \dfrac{1}{4}H(r)\beta^2I, \label{wRic_k_integral_estimates}
	\end{align}
	where $I = \int_{R_0 - a}^{R_0 - a + 1} e^{2\varphi} \; dt$.\\
	Using the relationship between $\nabla$ and $\nablatilde$ we have
	\begin{align}
		\notag \sum_{i = 1}^k d\varphi(\sigmadot(0))e^{2\varphi(q)} + \sum_{i = 1}^k g\left(\nabla_{e^{\varphi(q)} e_i}e^{\varphi(q)} e_i, \sigmadot(0) \right) =& \sum_{i = 1}^k d\varphi(\sigmadot(0))e^{2\varphi(q)} + \sum_{i = 1}^k 2e^{2\varphi(q)}d\varphi(e_i)\alpha\\
		\notag&- \sum_{i = 1}^k d\varphi(\sigmadot(0))e^{2\varphi(q)}\\
		=& 2e^{2\varphi(q)}\alpha \sum_{i = 1}^k d\varphi(e_i)|_q. \label{Sum_of_derivatives_estimates}
	\end{align}
	Since $\sum_{i = 1}^k \int_0^{T(r)}g\left(D_t\left(\left(1 - \dfrac{t}{T(r)}\right)e^\varphi e_i\right), \sigmadot\right)^2 \; dt \ge 0$, we will discard it. Summarizing from above, by (\ref{After_Use_2nd_Var_Formula_wRic_k}), (\ref{Integral_of_||^2=-H(r)I(r)/W(r)}), (\ref{Sum_wsec_estimates_using_wsec_and_wRic}), (\ref{wsec_integral_estimates}), (\ref{wRic_k_integral_estimates}), and \ref{Sum_of_derivatives_estimates}, we obtain a lower bound for $\sum_{i = 1}^k \Hess_{\gtilde} f(e_i, e_i)$ as
	\begin{align*}
		\sum_{i = 1}^k \Hess_{\gtilde} f(e_i, e_i) \ge& -\dfrac{1}{8}H(r)I - k\Kbar(r)\left(\alpha^2 + 4\alpha\beta\right) \left(\dfrac{2W_1T(r)^{j+1}}{(j+1)(j+2)(j+3)} + \dfrac{W_2T(r)}{3}\right)\\
		& + \dfrac{1}{4}H(r)\beta^2I + 2e^{2\varphi(q)}\alpha\sum_{i = 1}^kd\varphi(e_i)|_q.
	\end{align*}
	Note that $2e^{2\varphi(q)}\alpha\sum_{i = 1}^k d\varphi(e_i)|_q$ is always bounded since each $d\varphi(e_i)|_q$ is bounded. We can denote its lower bound by $2e^{2\varphi(q)}\alpha P(r)$ for some $P(r) \in \R$. Also, for simplicity, denote $F(r) = k\Kbar(r)\left(\dfrac{2W_1T(r)^{j+1}}{(j+1)(j+2)(j+3)} + \dfrac{W_2T(r)}{3}\right)$. Combining results from above, we get
	\begin{align}
		\notag &Z(r)\sum_{i = 1}^k |e_jf|^2 + \sum_{i = 1}^k \Hess_{\gtilde} f(e_i, e_i)\\
		\notag \ge& Z(r)\alpha^2e^{2\varphi(q)} - \dfrac{1}{8}H(r)I - F(r)\left(\alpha^2 + 4\alpha\beta\right) + \dfrac{1}{4}H(r)\beta^2I + 2e^{2\varphi(q)}\alpha P(r)\\
		\notag =& Z(r)\alpha^2e^{2\varphi(q)} - \dfrac{1}{8}H(r)I - F(r)\left(\alpha^2 + 4\alpha\beta\right) + \dfrac{1}{4}H(r)I - \dfrac{1}{4}H(r)\alpha^2I + 2e^{2\varphi(q)}\alpha P(r)\\
		\notag \ge& \left(Z(r)e^{2\varphi(q)} - F(r) - \dfrac{1}{4}H(r)I\right)\alpha^2 - 2\alpha(2\beta F(r) + e^{2\varphi(q)}P(r)) + \dfrac{1}{4}H(r)I - \dfrac{1}{8}H(r)I\\
		\notag =& \left(\alpha\sqrt{Z(r)e^{2\varphi(q)} - F(r) - \dfrac{1}{4}H(r)I} - \dfrac{2\beta F(r) + e^{2\varphi(q)}P(r)}{\sqrt{Z(r)e^{2\varphi(q)} - F(r) - \frac{1}{4}H(r)I}}\right)^2\\
		&- \left(\dfrac{2\beta F(r) + e^{2\varphi(q)}P(r)}{\sqrt{Z(r)e^{2\varphi(q)} - F(r) - \frac{1}{4}H(r)I}}\right)^2 + \dfrac{1}{4}H(r)I - \dfrac{1}{8}H(r)I \label{Complete_the_square}\\
		\ge& - \left(\dfrac{2\beta F(r) + e^{2\varphi(q)}P(r)}{\sqrt{Z(r)e^{2\varphi(q)} - F(r) - \frac{1}{4}H(r)I}}\right)^2 + \dfrac{1}{8}H(r)I,
	\end{align}
	where step (\ref{Complete_the_square}) is done by completing the square.\\
	Notice that $Z(r)$ only depends on $H(r), \Kbar(r)$, and $T(r)$. After choosing $Z(r)$ and $W(r)$ sufficiently large, one obtains
	$$- \left(\dfrac{2\beta F(r) + e^{2\varphi(q)}P(r)}{\sqrt{Z(r)e^{2\varphi(q)} - F(r) - \frac{1}{4}H(r)I}}\right)^2 + \dfrac{1}{8}H(r)I \ge \dfrac{1}{16}H(r)I.$$
	Since $f$ supports $b_p$ at $q$, it follows that
	\begin{align*}
		\Hess_{\gtilde}(\chi \circ  f)(e^{\varphi(q)} e_i, e^{\varphi(q)} e_i) &= \left(Z(r)\sum_{i = 1}^k |e^\varphi e_i f|^2 + \sum_{i = 1}^k \Hess_{\gtilde} f (e^\varphi e_i, e^\varphi e_i)\right)(\chi' \circ f)(q);\\
		e^{2\varphi(q)}\Hess_{\gtilde}(\chi \circ  f)(e_i, e_i) &\ge \dfrac{1}{16}H(r)I(\chi' \circ b_p)(q);\\
		\Hess_{\gtilde}(\chi \circ  f)(e_i, e_i) &\ge \dfrac{1}{16}H(r)Ie^{-2\varphi(q)}\chi'(r).
	\end{align*}
	Since $(\chi \circ f)''$ is bounded on the compact set $\overline{B}_q(T(r))$, we can also define a continuous function $A(r) \ge \dfrac{\Hess_{\gtilde}(\chi \circ f)(e^\varphi v, e^\varphi v)}{e^{2\varphi}|v|^2} = \Hess_{\gtilde}(\chi \circ f)\left(e_v, e_v\right)$, where $e_v$ is the unit vector pointing to the direction of $v$. Then for all $v \in T_qM$,
	$$|\Hess_{\gtilde}(\chi \circ f)(v, v)| \le A(r)|v|^2.$$
	By Lemma \ref{Bilinear_form_lemma}, there exists some continuous function $\epsilon(r) > 0$ that only depends on $H(r)$ and $A(r)$ such that for all $\{v_1, ..., v_k\} \subseteq T_qM$ with $|g(v_i, v_j) - \delta_{ij}| < \epsilon(r)$, we have
	$$\sum_{i = 1}^k \Hess_{\gtilde}(\chi \circ  f)(v_i, v_i) \ge \dfrac{1}{32}H(r)Ie^{-2\varphi(q)}.$$
	Since $\chi \circ f$ supports $\chi \circ b_p$ at $q$, we know
	$$\sum_{i = 1}^k \Ctilde(\chi \circ  f)(q; v_i) \ge \dfrac{1}{32}H(r)Ie^{-2\varphi(q)},$$
	and hence $\chi \circ b_p \in \Ctilde(k)$.
\end{proof}
This immediately yields the result in Theorem \ref{Shen's_Result} as follows.
\begin{proof}[Proof of Theorem \ref{Shen's_Result}]
	The conclusion follows from Lemma \ref{chi_bp_in_Ctilde(k)} and Lemma \ref{Smoothing_with_proper_Morse_function}. Together, they imply the existence of a proper Morse function $F \in \Ctilde(k)$ such that $|F - \chi \circ b_p| < \epsilon$ for some $\epsilon > 0$. By Lemma \ref{chi_bp_in_Ctilde(k)}, the index of $F$ is at most $k-1$. The result then follows from in \cite[Theorem 3.5.]{Milnor2}.
\end{proof}
As a special case of the above result with $k = n-1$, we obtain the following corollary:
\begin{coro}\label{CoroFromShen}
	Let $(M,g,e^{-\varphi})$ be a proper open weighed Riemannian manifold such that $e^\varphi$ grows strictly sublinearly. Suppose $\Ric_{(n-1)\varphi}^{1-n}(M) \ge 0$ and $\Ric_{(n-1)\varphi}^{1-n}(M) > 0$ outside of a compact set. Then
	$$H_{n-1}(M; \Z) = H_n(M; \Z) = 0.$$
\end{coro}
This result can be compared with Theorem \ref{AliceLim'sResult}. It replaces the boundedness assumption on $\varphi$ with a growth condition on $e^\varphi$, which roughly corresponds to logarithmic growth of $\varphi$. This also comes at the additional cost of assuming that $M$ is proper.
\section{Compact Manifolds with $\wRic_{k,\varphi}(M) > 0$ and Weighted Convex Boundary}\label{CompactManifoldsWithwRickBounds}
To establish Theorem \ref{WuThm2}, we use the following lemma from \cite{Wu}:
\begin{lem}\label{Ctidle(k)Lemma}
	A locally Lipschitz function $f \in \Ctilde(k)$ if for any $x_0 \in M$, there exists a neighborhood $B$ of $x_0$ such that for each $x \in W$, there exists some $C^\infty$ function $g_x$ supporting $f$ at $x$ which has the following two properties:
	\begin{enumerate}
		\item For some $\eta > 0$, $\sum_{j = 1}^k \Hess_{\gtilde} g_x (e^\varphi X_i,e^\varphi X_i) \ge \eta$, for every orthonormal set $\{e^\varphi X_1, ..., e^\varphi X_k\}$ in $T_xM$.
		\item There exists some constants $A_1, A_2 > 0$ such that $-A_1 \le \Hess_{\gtilde} g_x (e^\varphi X, e^\varphi X) \le A_2$ for all $x \in B$ and for all orthonormal vectors $e^\varphi X \in T_xM$.
	\end{enumerate}
\end{lem}
%
With minor modifications to the arguments in Theorem \ref{WuThm1} and Theorem \ref{Shen's_Result}, we can prove Theorem \ref{WuThm2} as follows.
\begin{proof}[Proof to Theorem \ref{WuThm2}:]
	Let $x \in M$ and let $\rho(y)$ denote the distance from $x$ to $y \in \partial M$, and let $\rho_0$ be the maximal distance to the boundary. Let $B$ be a geodesic ball centered at $x$ such that $B \subseteq \rho^{-1}([\rho_1,\rho_2])$, for some $0<\rho_1<\rho_2<\rho_0$.\\
	Let $x \in B$ and $y \in \partial M = N$. Let $\gamma : [0,b] \longrightarrow M$ be a minimizing normal $\nabla$-geodesic from $x$ to $y$. Let $\{E_1, ..., E_k\} \subseteq T_xM$ be $g$-orthonormal vectors. For every $1 \le j \le k$, decompose $E_j = \alpha_j E_j^\perp + \beta_j \gammadot(0)$, where $|E_j^\perp|_g = 1$ and $E_j^\perp \perp \gammadot(0)$. Let $E_j(t)$ be a $\nabla$-parallel transport of $E_j$ along $\gamma(t)$ for every $1 \le j \le k$. Define
	$$W_j(t) = \alpha_j(t)E_j^\perp(t) + \beta_j(t)\gammadot(t),$$
	where $\alpha_j(t)$ and $\beta_j(t)$ are differentiable and $\alpha_j^2(t) + \beta_j^2(t) = 1$ for all $0 \le t \le b$. Also, suppose $\begin{cases}
		\beta_j(0) = \beta_j\\
		\beta_j(b) = 0
	\end{cases}$ and $\begin{cases}
	\alpha_j(0) = \alpha_j\\
	\alpha_j(b) = 1.
	\end{cases}$ This means $|W_j(t)| = 1$ for all $t$. Define variations
	$$\theta_j(t,s) = \wexp(se^\varphi W_j(t))$$
	Define a function $f$ locally at $x$ by
	$$\left(f \circ \wexp_{\gamma(t)}\right)(E_j) = \int_0^b \left|\dfrac{\partial \theta_j}{\partial t}(t,s)\right| \, dt.$$
	Let $\alpha(t),\beta(t) > 0$ satisfy $\alpha^2(t) + \beta^2(t) = 1$ and
	$$\beta^2 = \sum_{j = 1}^k g\left(E_j,\gammadot(t)\right)^2 = \sum_{i = 1}^k \beta_j^2(t).$$
	Define $\chi$ as in (\ref{Chi_in_WuThm1}) and set $h = \chi \circ (-f)$. At $t = 0$, $W_j(0) = E_j$, so it suffices to show
	$$\sum_{j = 1}^k \Hess_{\gtilde} h(e^\varphi W_j, e^\varphi W_j) \ge \eta$$
	for some $\eta > 0$.\\
	By the first variation formula,
	\begin{equation}
		Z(-b)\sum_{j = 1}^kdf(e^\varphi W_j)^2 = Z(-b)ke^{2\varphi}\beta^2. \label{Sum_of_Z(b)df^2}
	\end{equation}
	By the second variation formula from Proposition \ref{Second_Variation_Formula},
	\begin{align}
		\notag \sum_{j=1}^k -\Hess_{\gtilde} f(e^\varphi W_j, e^\varphi W_j) &= \sum_{j=1}^k \int_0^b - \big|D_t e^\varphi W_j - d\varphi(\dot\gamma) e^\varphi W_j\big|^2 \; dt \Big|_{s=0}\\
		\notag &+ \sum_{j=1}^k \int_0^b g\big(R^\varphi(e^\varphi W_j, \dot\gamma)\dot\gamma, e^\varphi W_j\big) + g(D_t e^\varphi W_j, \dot\gamma)^2 \; dt \Big|_{s=0}\\
		&+ \sum_{j=1}^k \Big(d\varphi(\dot\gamma)|e^\varphi W_j|^2 + g(D_s e^\varphi W_j, \dot\gamma)\Big) \Big|_{t=b,\, s=0}^{t=0,\, s=0} \label{Sum_Hess_gtilde(e^varphi_W_j,e^varphi_W_j)_estimate}
	\end{align}
	Since $E_j(t)$ is a $\nabla$-parallel transport of $E_j$ along $\gamma(t)$, we have
	\begin{align}
		\notag &\sum_{j=1}^k \int_0^b - \big|D_t e^\varphi W_j - d\varphi(\gammadot)e^\varphi W_j \big|^2 \; dt \Big|_{s=0}\\
		\notag =& \sum_{j=1}^k \int_0^b - \big|d\varphi(\gammadot)e^\varphi W_j + e^\varphi \nabla_{\gammadot}W_j - d\varphi(\gammadot)e^\varphi W_j \big|^2 \; dt \Big|_{s=0}\\
		\notag =& \sum_{j=1}^k \int_0^b -e^{2\varphi} \big|\nabla_{\gammadot}W_j(t)\big|  \; dt \Big|_{s=0}\\
		\notag =& \sum_{j=1}^k \int_0^b -e^{2\varphi} \big|\nabla_{\gammadot}(\alpha_j(t)E_j^\perp(t) + \beta_j(t)\gammadot(t))\big| \; dt \Big|_{s=0}\\
		=& \sum_{j=1}^k \int_0^b -e^{2\varphi}(\alpha_j'(t)^2 + \beta_j'(t)^2)  \; dt. \label{Sum_int_0^b||^2dt}
	\end{align}
	By Lemma (\ref{Estimate_sum_wsec}), we have
	\begin{align}
		\notag \sum_{j=1}^k \int_0^b g\big(R^\varphi(e^\varphi W_j, \dot\gamma)\dot\gamma, e^\varphi W_j\big) \; dt \Big|_{s=0} =& \sum_{j=1}^k \int_0^b e^{2\varphi}g\big(R^\varphi(W_j, \dot\gamma)\dot\gamma, W_j\big) \; dt \Big|_{s=0}\\
		\notag \ge& \int_0^b - e^{2\varphi}k\Kbar(\beta^2(t) + 4\alpha(t)\beta(t)) \; dt\\
		\notag &+ \int_0^b e^{2\varphi}\min\{\wRic_{k,\varphi}(\gamma(t))\}\alpha^2(t) \; dt\\
		\notag \ge& -\umax^2k\Kbar\max\{\beta^2(t)+4\alpha(t)\beta(t)\}\rho_0\\
		&+ \int_0^b e^{2\varphi}\min\{\wRic_{k,\varphi}(\gamma(t))\}\alpha^2(t) \; dt, \label{Sum_of_curvatures}
	\end{align}
	where $\Kbar$ is the maximum of $|\wsec_\varphi(M)|$, which can be attained because $M$ is compact. $\umax^2$ is the maximum value of $e^{2\varphi}$ on $M$, and $\rho_0$ is the maximal distance from $N$.\\
	Also, 
	\begin{align}
		\notag& \sum_{j=1}^k \int_0^b g(D_t e^\varphi W_j, \dot\gamma)^2 \; dt \Big|_{s=0}\\
		\notag =& \sum_{j=1}^k \int_0^b g( d\varphi(\gammadot)e^\varphi W_j + e^\varphi \nabla_{\gammadot} W_j, \dot\gamma)^2 \; dt \Big|_{s=0}\\
		\notag =& \sum_{j=1}^k \int_0^b g( d\varphi(\gammadot)e^\varphi (\alpha_j(t)E_j^\perp(t) + \beta_j(t)\gammadot(t)) + e^\varphi \nabla_{\gammadot}(\alpha_j(t)E_j^\perp(t) + \beta_j(t)\gammadot(t)), \dot\gamma)^2 \; dt\\
		\notag =& \sum_{j=1}^k \int_0^b g( d\varphi(\gammadot)e^\varphi\beta_j(t)\gammadot(t) + e^\varphi (\alpha_j'(t)E_j^\perp(t) + \beta_j'(t)\gammadot(t)), \dot\gamma)^2 \; dt\\
		\notag =& \sum_{j=1}^k \int_0^b g( d\varphi(\gammadot)e^\varphi\beta_j(t)\gammadot(t) + e^\varphi\beta_j'(t)\gammadot(t), \dot\gamma)^2 \; dt\\
		=& \sum_{j=1}^k \int_0^b d\varphi(\gammadot)^2e^{2\varphi}\beta_j^2(t) + e^{2\varphi}\beta_j'(t)^2 \; dt\\
		\ge& \sum_{j=1}^k \int_0^b e^{2\varphi}\beta_j'(t)^2 \; dt. \label{Sum_of_extra_term}
	\end{align}
	At $t=0, s=0$, by the relationship between $\nablatilde$ and $\nabla$, we have
	\begin{align}
		\notag \sum_{j=1}^k \Big(d\varphi(\dot\gamma)|e^\varphi W_j|^2 + g(D_s e^\varphi W_j, \dot\gamma)\Big) \Big|_{t=0,\, s=0} =& \sum_{j=1}^k \Big(d\varphi(\dot\gamma)e^{2\varphi(x)} + g(\nabla_{e^{\varphi(x)} W_j}e^{\varphi(x)} W_j, \dot\gamma)\Big)\\
		\notag =& \sum_{j=1}^k \left(d\varphi(\dot\gamma)e^{2\varphi(x)} + g\left(\nablatilde_{e^{\varphi(x)} W_j} e^{\varphi(x)} W_j, \gammadot\right)\right)\\
		\notag&+ \sum_{j=1}^k 2e^{2\varphi(x)}d\varphi(W_j)g\left(W_j, \gammadot\right)\\
		\notag&- \sum_{j=1}^k d\varphi(\gammadot)\left|e^{\varphi(x)} W_j\right|^2\\
		\notag =& \sum_{j=1}^k \left(d\varphi(\dot\gamma)e^{2\varphi(x)} + 2e^{2\varphi(x)}d\varphi(W_j)g\left(W_j, \gammadot\right)\right)\\
		\notag &- \sum_{j=1}^k d\varphi(\gammadot)\left|e^{\varphi(x)} W_j\right|^2\\
		\notag =& \sum_{j=1}^k \Big(d\varphi(\dot\gamma)e^{2\varphi(x)} + 2e^{2\varphi(x)}\beta_j d\varphi(W_j)\Big)\\
		\notag &- \sum_{j=1}^k d\varphi(\gammadot)e^{2\varphi(x)}\\
		=& \sum_{j=1}^k  2e^{2\varphi(x)}\beta_j d\varphi(W_j). \label{Sum_of_at_0}
	\end{align}
	At $t=b, s=0$, we have
	\begin{align}
		\notag \sum_{j=1}^k \Big(d\varphi(\dot\gamma)|e^\varphi W_j|^2 + g(D_s e^\varphi W_j, \dot\gamma)\Big) \Big|_{t=b,\, s=0} =& \sum_{j=1}^k e^{2\varphi(y)} d\varphi(\dot\gamma)|\alpha_j(b)E_j^\perp(b)|^2\\
		\notag &+ \sum_{j=1}^k g(\nabla_{e^{\varphi(y)}\alpha_j(b)E_j^\perp(b)}e^{\varphi(y)}\alpha_j(b)E_j^\perp(b), \dot\gamma)\\
		=& -\sum_{j=1}^k e^{2\varphi(y)}\alpha_j^2(b)\IItilde(E_j^\perp, E_j^\perp). \label{Sum_of_at_b=sum_of_IItilde}
	\end{align}
	Therefore, combining (\ref{Sum_Hess_gtilde(e^varphi_W_j,e^varphi_W_j)_estimate}), (\ref{Sum_int_0^b||^2dt}), (\ref{Sum_of_curvatures}), (\ref{Sum_of_extra_term}), (\ref{Sum_of_at_0}), and (\ref{Sum_of_at_b=sum_of_IItilde}), we have
	\begin{align}
		\notag \sum_{j=1}^k -\Hess_{\gtilde} f(e^\varphi W_j, e^\varphi W_j) \ge& \sum_{j=1}^k \int_0^b -e^{2\varphi}(\alpha_j'(t)^2 + \beta_j'(t)^2) + e^{2\varphi}\beta_j'(t)^2 \; dt\\
		\notag &-\umax^2k\Kbar\max\{\beta^2(t)+4\alpha(t)\beta(t)\}\rho_0\\
		\notag &+ \int_0^b e^{2\varphi}\min\{\wRic_{k,\varphi}(\gamma(t))\}\alpha^2(t) \; dt\\
		&+ \sum_{j=1}^k  2e^{2\varphi(x)}\beta_j d\varphi(W_j) + \sum_{j=1}^k e^{2\varphi(y)}\alpha_j^2(b)\IItilde(E_j^\perp, E_j^\perp) \label{Sum_of_negative_Hess_gtilde}
	\end{align}
	Combining (\ref{Sum_of_negative_Hess_gtilde}) and (\ref{Sum_of_Z(b)df^2}), we have
	\begin{align}
		\notag \sum_{j=1}^k \Hess_{\gtilde} h(e^\varphi E_j, e^\varphi E_j) \ge& Z(-b)ke^{2\varphi}\beta^2 + \sum_{j=1}^k \int_0^b -e^{2\varphi}(\alpha_j'(t)^2 + \beta_j'(t)^2) + e^{2\varphi}\beta_j'(t)^2 \; dt\\
		\notag &-\umax^2k\Kbar\max\{\beta^2(t)+4\alpha(t)\beta(t)\}\rho_0 + \sum_{j=1}^k  2e^{2\varphi(x)}\beta_j d\varphi(W_j) \\
		\notag &+ \int_0^b e^{2\varphi}\min\{\wRic_{k,\varphi}(\gamma(t))\}\alpha^2(t) \; dt + \sum_{j=1}^k e^{2\varphi(y)}\alpha_j^2(b)\IItilde(E_j^\perp, E_j^\perp)\\
		\notag \ge& Z(-b)k\umin^2\beta^2 - k\umax^2\max_{1 \le j \le k}\{\alpha_j'(t)^2\} \rho_0\\
		\notag &-\umax^2k\Kbar\max\{\beta^2(t)+4\alpha(t)\beta(t)\}\rho_0 + \sum_{j=1}^k  2e^{2\varphi(x)}\beta_j d\varphi(W_j)\\
		&+ \int_0^b e^{2\varphi}\min\{\wRic_{k,\varphi}(\gamma(t))\}\alpha^2(t) \; dt + \sum_{j=1}^k e^{2\varphi(y)}\alpha_j^2(b)\IItilde(E_j^\perp, E_j^\perp),
	\end{align}
	where $\umin^2$ is the minimum value of $e^{2\varphi}$ on $M$.\\
	Choosing $Z(-b))$ sufficiently enough, we obtain
	\begin{align*}
		\eta \le& Z(-b)k\umin^2\beta^2 - k\umax^2\max_{1 \le j \le k}\{\alpha_j'(t)^2\} \rho_0-\umax^2k\Kbar\max\{\beta^2(t) +4\alpha(t)\beta(t)\}\rho_0\\
		&+ \sum_{j=1}^k  2e^{2\varphi(x)}\beta_j d\varphi(W_j).
	\end{align*}
	Together with the assumption that $\wRic_{k,\varphi}(M) \ge 0$ and $\Lambdatilde_k \ge 0$, we know
	\begin{align}
		\notag \sum_{j=1}^k \Hess_{\gtilde} h(e^\varphi E_j, e^\varphi E_j) \ge& \int_0^b e^{2\varphi}\min\{\wRic_{k,\varphi}(\gamma(t))\}\alpha^2(t) \; dt + \sum_{j=1}^k e^{2\varphi(y)}\alpha_j^2(b)\IItilde(E_j^\perp, E_j^\perp)\\
		\notag &+ \eta\\
		\notag \ge& \umin^2 \alpha^2(0) \rho_0 \wRic_{k,\varphi}(M) + \umin^2\Lambdatilde_k + \eta\\
		\ge& \eta.
	\end{align}
	Note that $h \in C^2(M)$. By restricting to a smaller compact neighborhood $B' \subset B$, the eigenvalues of $\Hess_{\gtilde} h$ are uniformly bounded, giving
	$$-A_1 \le \Hess_{\gtilde} h(e^\varphi X, e^\varphi X) \le A_2.$$
	for some positive $A_1, A_2$ and for all orthonormal $e^\varphi X \in T_xM$. Hence $h \in C^2(M)$ satisfies the required conditions.\\
	Moreover, $\chi(-f)$ supports $\chi(-\rho)$ at $x$ since $-f$ supports $-\rho$ and $\chi' > 1$. By Lemma \ref{Ctidle(k)Lemma}, this implies $\chi(-\rho) \in \Ctilde(k)$.\\
	By Lemma \ref{Ctidle(k)Lemma}, this shows $\chi(-\rho) \in \Ctilde(k)$.\\
	The rest of the proof follows from Lemma \ref{Smoothing_with_proper_Morse_function} and Morse Theory. Let $c > 0$ be small enough such that $T(c) = \{z \in M \; | \; (\chi \circ -\rho)(z) \le c\}$ has no critical points for $\chi \circ -\rho$. Then $T(c)$ deformation retracts onto $\partial M$, so $M' = M \setminus T(c)$ has the same homotopy type as $M$ by Theorem 3.1 in \cite{Milnor2}. By Lemma \ref{Smoothing_with_proper_Morse_function}, there exists a proper Morse function $F \in \Ctilde(k)$ with $|F - \chi(-\rho)| < \epsilon$, and hence $\mathrm{index}(F) \le k-1$. \cite[Theorem 3.5.]{Milnor2} then gives (a) in Theorem \ref{WuThm2}.\\
\end{proof}
With slight modification, the same argument also yields the following corollaries.
\begin{coro}
	Let $(M,g,e^{-\varphi})$ be an $n$-dimensional weighted Riemannian manifold. Let $1 \le k \le n$ and suppose $\wRic_{k,\varphi}(M) \ge 0$. Let $N \subseteq M$ be a $\gtilde$-totally geodesic hypersurface. Then $M$ has the same homotopy type of a CW complex of obtained from $N$ by attaching a finite numbers of cells of dimension $\ge n-k+1$.
\end{coro}
The proof is identical to that of Theorem \ref{WuThm2}, with $N$ replacing $\partial M$ and $\rho$ taken as the distance to $N$. Since $N$ is $\gtilde$-totally geodesic, we have $\Lambdatilde_k = 0$.
\begin{coro}
	Let $(M,g,e^{-\varphi})$ be an $n$-dimensional simply connected compact weighted Riemannian manifold, and suppose $\wsec_\varphi(M) > 0$. Let $N \subseteq M$ be an orientable $\gtilde$-totally geodesic compact hypersurface such that the $\wsec_\varphi(M) > 0$ in a neighborhood of $N$. Then the pair $(M,N)$ is homeomorphic to $(S^n, S^{n-1})$, where $S^{n-1}$ is the equatorial hypersphere.
\end{coro}
The proof of the unweighted version of this Corollary is demonstrated in \cite[Theorem 4.]{Wu}, and the weighted version has a similar proof.

\section{Frankel's Theorem with $\wRic_{k,\varphi}(M) > 0$}\label{FrankelsTheorem}
The proof of the weighted version of Frankel’s Theorem with $\wRic_{k,\varphi}(M) > 0$ is also a direct application of the second variation formula for arc length stated in Proposition \ref{Second_Variation_Formula}.
\begin{proof}[Proof of \ref{Frankel's_Theorem}:]
	Let $p_1 \in N_1, p_2 \in N_2$ be points realizing the minimal distance between $N_1$ and $N_2$.\\
	If $p_1 \not= p_2$, let $\gamma : [0, l] \longrightarrow M$ be a unit speed geodesic from $p_1$ to $p_2$. Parallel transport $T_{p_1}N_1$ along $\gamma$ to $p_2$, and denote the resulting subspace by $T'_{p_1}N_1$. Since $T_{p_1}N_1 \perp \gammadot$ at $p_1$, it follows that $T'_{p_1}N_1 \perp \gammadot$ at $p_2$. Hence,
	$$\dim(T'_{p_1}N_1 \cap T_{p_2}N_2) \ge r + s - (n-1) \ge k.$$
	Choose $k$ orthonormal vectors $E_1',...,E_k' \in T'_{p_1}N_1 \cap T_{p_2}N_2$, and parallel transport them back to $p_1$ via $\gamma$, denoting the resulting vectors by $E_1, ..., E_k$. Now, consider variations $$\theta_j(t,s) = \widetilde{\exp}_{\gamma(t)}(se^\varphi E_j(t)),$$
	for $1 \le j \le k$.\\
	At $p_1$, i.e. $t = 0$, we have
	$$d\varphi(\gammadot(0))|e^{\varphi(p_1)}E_j(0)|^2 = e^{2\varphi(p_1)}d\varphi(\gammadot(0)),$$
	and
	\begin{align*}
		g(D_s(e^{\varphi} E_j(0)), \gammadot(0)) &= g((D_se^\varphi)E_j(0), \gammadot(0)) + g(e^\varphi \nabla_{e^\varphi E_j(0)}E_k(0), \gammadot(0))\\
		&= e^{2\varphi(p_1)} g(\nabla_{E_j(0)}E_j(0), \gammadot(0)).
	\end{align*}
	At $p_2$, i.e. $t = 1$, similarly,
	$$d\varphi(\gammadot(1))|e^{\varphi(p_2)}E_j(1)|^2 = e^{2\varphi(p_2)}d\varphi(\gammadot(1)),$$
	and
	$$g(D_s(e^{\varphi} E_j(1)), \gammadot(1)) = g((D_se^\varphi)E_j(1), \gammadot(1)) + g(e^\varphi \nabla_{e^\varphi E_j(1)}E_k(1), \gammadot(1)) = e^{2\varphi(p_2)} g(\nabla_{E_j(1)}E_j(1), \gammadot(1)).$$
	Applying the second variation formula for arc length from Proposition \ref{Second_Variation_Formula}, we obtain
	\begin{align*}
		\sum_{j = 1}^{k} \dfrac{d^2}{ds^2}\Bigg|_{s = 0} L_g(\theta_j(t,s)) =& \sum_{j = 1}^k \int_0^1 \big|D_t (e^\varphi E_j(t)) - d\varphi(\dot\gamma)(e^\varphi E_j(t))\big|^2 \, dt \Big|_{s=0}\\
		&- \sum_{j = 1}^k \int_0^1 g\big(R^\varphi(e^\varphi E_j(t), \dot\gamma)\dot\gamma, e^\varphi E_j(t)\big) \, dt \Big|_{s=0}\\
		&+ \sum_{j = 1}^{k} \Big(d\varphi(\dot\gamma)|e^\varphi E_j(t)|^2 + g(D_s (e^\varphi E_j(t)), \dot\gamma)\Big) \Big|_{t=0,\, s=0}^{t=1,\, s=0}\\
		=& -\int_0^1 e^{2\varphi} \sum_{j = 1}^{k} \wsec(E_j(t),\gammadot) \, dt\\
		&+ \sum_{j=1}^{k} e^{2\varphi(p_2)}\IItilde(E_j(1), E_j(1)) - \sum_{j=1}^{k} e^{2\varphi(p_1)}\IItilde(E_j(0), E_j(0)).\\
		=& -\int_0^1 e^{2\varphi} \sum_{j = 1}^{k} \wsec(E_j(t),\gammadot) \, dt\\
		< & 0.
	\end{align*}
	This contradicts the fact that $\gamma$ is length minimizing. Therefore
	$$N_1 \cap N_2 \not= \varnothing.$$
\end{proof}
This immediately yields the following corollary, corresponding to the case $k = n-1$.

\begin{coro}
	In a complete connected weighted Riemannian manifold with $\Ric_{(n-1)\varphi}^{1-n} > 0$, any two minimal hypersurfaces with respect to $\nablatilde$ must intersect.
\end{coro}

\begin{coro}[Wilking's Connectedness Lemma]
	Let $(M^n, g, e^{-\varphi})$ be a compact weighted Riemannian manifold with $\wRic_{k,\varphi}(M) \ge 0$ for some $1 \le k \le n-1$.  Suppose $N^{k-n} \subseteq M^n$ is a compacted, totally geodesic with respect to $\gtilde$ embedded submanifold of codimension $k$. Then the inclusion $N \longrightarrow M$ is $(n-2k+1)$-connected.
\end{coro}

The proof of this theorem is the same as its original in \cite{Wilking}, as it follows a very similar argument to the proof of Theorem \ref{Frankel's_Theorem}.
\bibliographystyle{amsalpha}
\bibliography{Reference}
\end{document}